\documentclass{amsart}
\usepackage[all,ps,cmtip]{xy}
\usepackage{amsmath, amssymb}
\usepackage{amscd}
\usepackage{tikz,color}

\newtheorem{theorem}{Theorem}[section]
\newtheorem{lemma}[theorem]{Lemma}

\newtheorem{corollary}[theorem]{Corollary}

\newtheorem{proposition}[theorem]{Proposition}
\newtheorem{conjecture}[theorem]{Conjecture}

\theoremstyle{definition}
\newtheorem{definition}[theorem]{Definition}

\theoremstyle{remark}
\newtheorem{remark}[theorem]{Remark}

\numberwithin{equation}{section}

\DeclareMathOperator{\rank}{rk}

\DeclareMathOperator{\Hom}{Hom}

\DeclareMathOperator{\Coh}{Coh}

\DeclareMathOperator{\D}{D}

\DeclareMathOperator{\K}{K}

\DeclareMathOperator{\NS}{NS}

\DeclareMathOperator{\ch}{ch}

\DeclareMathOperator{\Ext}{Ext}

\makeatletter \@namedef{subjclassname@2020}{\textup{2020}
Mathematics Subject Classification} \makeatother

\begin{document}
\title{Stability conditions on fibred threefolds}

\author{Hao Max Sun}
\address{Department of Mathematics, Shanghai Normal University, Shanghai 200234, People's Republic of China}

\email{hsun@shnu.edu.cn, hsunmath@gmail.com}




\date{February 1, 2022}


\begin{abstract}
We give a conjectural construction of Bridgeland stability
conditions on the derived category of fibred threefolds. The
construction depends on a conjectural Bogomolov-Gieseker type
inequality for certain stable complexes. It can be considered as a
relative version of the construction of Bayer, Macr\`i and Toda. We
prove the conjectural Bogomolov-Gieseker type inequality in the case
of relative projective planes over curves. This gives the the
existence of Bridgeland stability conditions on such threefolds.
\end{abstract}

\maketitle

\setcounter{tocdepth}{1}

\tableofcontents

\section{Introduction}
Stability conditions
for triangulated categories were introduced by Bridgeland
in \cite{Bri1}. Since then, they have drawn a lot of attentions, and have
been investigated intensively. The existence of stability conditions
on three-dimensional varieties is often considered the biggest open
problem in the theory of Bridgeland stability conditions. In \cite{BMT, BMS, BMSZ}, the authors introduced a conjectural
construction of Bridgeland stability conditions for any projective
threefold. Here the problem was reduced to proving a
Bogomolov-Gieseker type inequality for the third Chern character of
tilt-stable objects. It has been shown to hold for Fano 3-folds
\cite{BMSZ, Piy}, abelian 3-folds \cite{
BMS}, some product type threefolds \cite{Kos}, quintic threefolds \cite{Li2}, threefolds with vanishing Chern classes \cite{Sun}, etc. Recently, Yucheng Liu \cite{Liu} showed the existence of
stability conditions on product type varieties by a different method.

In this paper, we give a conjectural construction of Bridgeland stability
conditions on fibred threefolds. The
construction depends on a conjectural Bogomolov-Gieseker type
inequality for mixed tilt-stable complexes (Conjecture \ref{Conj1}).
We show that this conjecture gives the existence of stability conditions on fibred threefolds (Theorem \ref{main}).
Our construction can be considered as a
relative version of that of Bayer, Macr\`i and Toda \cite{BMT}. We
prove the conjectural Bogomolov-Gieseker type inequalities in the
case of relative projective planes over curves. This gives the the
existence of stability conditions on such threefolds:

\begin{theorem}[=Corollary \ref{Cor}]
Let $E$ be a rank three vector bundle on a complex smooth projective
curve $C$. Then there exist locally finite stability conditions on
$\mathbb{P}(E)$.
\end{theorem}

Throughout this paper, we let $f:\mathcal{X}\rightarrow C$ be a
smooth projective morphism from a complex smooth projective variety
of dimension $n\geq2$ to a complex smooth projective curve. We
denote by $F$ the general fiber of $f$, and fix a nef and relative
ample divisor $H$ on $\mathcal{X}$.

We now give a sketch of our construction; the details will be given in Section \ref{S5}. We use the $\mu_C$-stability introduced in \cite{BLMNPS}
to construct a torsion pair in $\Coh(\mathcal{X})$.
Let $\mathcal{T}_{\beta H}\subset\Coh(\mathcal{X})$ be the category generated by $\mu_C$-stable sheaves
of slope $\mu_C>\beta$ via extension. Similarly, let $\mathcal{F}_{\beta H}$ be the subcategory generated by $\mu_C$-stable sheaves
of slope $\mu_C\leq\beta$.
We define $\Coh^{\beta H}_C(\mathcal{X})\subset\D^b(\mathcal{X})$ as a tilt with respect to the torsion pair $(\mathcal{T}_{\beta H},\mathcal{F}_{\beta H})$:
$$\Coh_C^{\beta H}(\mathcal{X})=\langle\mathcal{T}_{\beta H}, \mathcal{F}_{
\beta H}[1]\rangle.$$
Different from the torsion pair defined via the classical slope-stability, there are torsion sheaves in $\mathcal{F}_{\beta H}$.
We
prove that there is a double-dual operation on $\Coh_C^{ \beta
H}(\mathcal{X})$ as well as for coherent sheaves (Lemma \ref{dual}).
For any $(\alpha,\beta, t)\in
\sqrt{\mathbb{Q}_{>0}}\times\mathbb{Q}\times\mathbb{Q}_{\geq0}$, we then define the following function on $\Coh^{\beta H}_C(\mathcal{X})$:
\begin{eqnarray*}
\nu_{\alpha,\beta,t}(\mathcal{E})=
\frac{(H^{n-2}
+tFH^{n-3})\ch_2^{\beta}(\mathcal{E})-\frac{t+1}{2}\alpha^2FH^{n-1}\ch_0(\mathcal{E})}{FH^{n-2}\ch^{\beta}_1(\mathcal{E})}.
\end{eqnarray*}
We show that it is a slope-function associated to a very weak stability condition, which we call mixed tilt-stability. Like tilt-stability, we
prove that mixed tilt-stable objects also satisfy a Bogomolov type inequality (Theorem \ref{Bog2}).
Using mixed tilt-stability,
we can define a torsion pair in $\Coh^{\beta H}_C(\mathcal{X})$ exactly as in the case of $\mu_C$-stability for $\Coh(\mathcal{X})$ above.
Tilting at this torsion pair produces a heart $\mathcal{A}_t^{\alpha,\beta}(\mathcal{X})$ of a t-structure. We prove that
$\mathcal{A}_t^{\alpha,\beta}(\mathcal{X})$ is noetherian
by the double-dual operation and
the Bogomolov type inequality.
Finally, Conjecture \ref{Conj1} guarantees
the positivity property for some central charge on $\mathcal{A}_t^{\alpha,\beta}(\mathcal{X})$ when $n=3$.

\subsection*{Organization of the paper}

Our paper is organized as follows. In Section \ref{S2}, we review
some basic notions and results of stability for
coherent sheaves on a fibred variety in \cite{BLMNPS}. Then in Section \ref{S3}, we recall the definition of very weak stability
conditions and give a relative version of the
tilt-stability constructed in \cite{BMT}. We will introduce the mixed tilt-stability and give its basic properties in Section \ref{S4}.
In Section \ref{S5}, we give the conjectural construction of Bridgeland stability
conditions on fibred threefolds and propose Conjecture \ref{Conj1}. We
prove this conjectural for relative projective planes over curves in Section \ref{S6}.

\subsection*{Notation}
Let $X$ be a smooth projective variety. We denote by $T_X$ and
$\Omega_X^1$ the tangent bundle and cotangent bundle of $X$,
respectively. $K_X$ and $\omega_X$ denote the canonical divisor and
canonical sheaf of $X$, respectively. We write $c_i(X):=c_i(T_X)$
for the $i$-th Chern class of $X$. We
write $\NS(X)$ for the N\'eron-Severi group of divisors up to
numerical equivalence. We also write $\NS(X)_{\mathbb{Q}}$,
$\NS(X)_{\mathbb{R}}$, etc. for $\NS(X)\otimes\mathbb{Q}$, etc. For
a triangulated category $\mathcal{D}$, we write $\K(\mathcal{D})$
for its Grothendieck group.

Let $\pi:\mathcal{X}\rightarrow S$ be a flat morphism of Noetherian
schemes and $W\subset S$ be a subscheme. We denote by
$\mathcal{X}_W=\mathcal{X}\times_S W$ the fiber of $\pi$ over $W$,
and by $i_W:\mathcal{X}_W\hookrightarrow \mathcal{X}$ the embedding
of the fiber. In the case that $S$ is integral, we write $K(S)$ for
its fraction field, and $\mathcal{X}_{K(S)}$ for the generic fiber
of $\pi$. We denote by $\D^b(\mathcal{X})$ the bounded derived
category of coherent sheaves on $\mathcal{X}$. Given $E\in
\D^b(\mathcal{X})$, we write $E_W$ (resp., $E_{K(S)}$) for the
pullback to $\mathcal{X}_W$ (resp., $\mathcal{X}_{K(S)}$).

Let $F$ be a coherent sheaf on $X$. We write $H^j(F)$ ($j\in
\mathbb{Z}_{\geq0}$) for the cohomology groups of $F$ and write
$\dim F$ for the dimension of its support. We write $\Coh_{\leq
d}(X)\subset\Coh(X)$ for the subcategory of sheaves supported in
dimension $\leq d$. Given a bounded t-structure on $\D^b(X)$ with
heart $\mathcal{A}$ and an object $E\in \D^b(X)$, we write
$\mathcal{H}_{\mathcal{A}}^j(E)$ ($j\in \mathbb{Z}$) for the
cohomology objects with respect to $\mathcal{A}$. When
$\mathcal{A}=\Coh(X)$, we simply write $\mathcal{H}^j(E)$. Given a
complex number $z\in\mathbb{C}$, we denote its real and imaginary
part by $\Re z$ and $\Im z$, respectively. We write
$\sqrt{\mathbb{Q}_{>0}}$ for the set $\{\sqrt{x}: x\in
\mathbb{Q}_{>0}\}$.


\subsection*{Acknowledgments}
I would like to thank Yunfeng Jiang and Yucheng Liu for useful discussions.
The author was supported by National Natural Science
Foundation of China (Grant No. 11771294, 11301201).

\section{Relative slope-stability}\label{S2}
We will review some
results in \cite{BLMNPS} and some basic notions of stability for
coherent sheaves.

For any $\mathbb{R}$-divisor $D$ on $\mathcal{X}$, we define the
twisted Chern character $\ch^{D}=e^{-D}\ch$. More explicitly, we
have
\begin{eqnarray*}
\begin{array}{lcl}
\ch^{D}_0=\ch_0=\rank  && \ch^{D}_2=\ch_2-D\ch_1+\frac{D^2}{2}\ch_0\\
&&\\
\ch^{D}_1=\ch_1-D\ch_0 &&
\ch^{D}_3=\ch_3-D\ch_2+\frac{D^2}{2}\ch_1-\frac{D^3}{6}\ch_0.
\end{array}
\end{eqnarray*}

The first important notion of stability for a sheaf is the relative
slope-stability. We define the relative slope $\mu_{H, F}$ of a
coherent sheaf $\mathcal{E}\in \Coh(\mathcal{X})$ by
\begin{eqnarray*}
\mu_{H, F}(\mathcal{E})= \left\{
\begin{array}{lcl}
+\infty,  & &\mbox{if}~\ch_0(\mathcal{E})=0,\\
&&\\
\frac{FH^{n-2}\ch_1(\mathcal{E})}{FH^{n-1}\ch_0(\mathcal{E})}, &
&\mbox{otherwise}.
\end{array}\right.
\end{eqnarray*}

\begin{definition}\label{slope}
A coherent sheaf $\mathcal{E}$ on $\mathcal{X}$ is $\mu_{H,
F}$-(semi)stable (or relative slope-(semi)stable) if, for all
non-zero subsheaves $\mathcal{F}\hookrightarrow \mathcal{E}$, we
have
$$\mu_{H, F}(\mathcal{F})<(\leq)\mu_{H, F}(\mathcal{E}/\mathcal{F}).$$
\end{definition}

Similarly, for any point $s\in C$, we can define
$\mu_{H_s}$-stability (or slope-stability) of a coherent sheaf
$\mathcal{G}$ on the fiber $\mathcal{X}_s$ over $s$ for the slope
$\mu_{H_s}$:
\begin{eqnarray*}
\mu_{H_s}(\mathcal{G})= \left\{
\begin{array}{lcl}
+\infty,  & &\mbox{if}~\ch_0(\mathcal{G})=0,\\
&&\\
\frac{H_s^{n-2}\ch_1(\mathcal{G})}{H_s^{n-1}\ch_0(\mathcal{G})}, &
&\mbox{otherwise}.
\end{array}\right.
\end{eqnarray*}
One sees that $\mu_{H, F}(\mathcal{E})=\mu_{H_s}(\mathcal{E}_s)$.

\begin{definition}\label{tor}
Let $\mathcal{A}_C$ be the heart of a $C$-local $t$-structure on
$\D^b(\mathcal{X})$ (see \cite[Definition 4.10]{BLMNPS}), and let
$\mathcal{E}\in \mathcal{A}_C$.
\begin{enumerate}
\item We say $\mathcal{E}$ is $C$-flat if $\mathcal{E}_c\in \mathcal{A}_c$
for every point $c\in C$, where $\mathcal{A}_c$ is the heart of the
$t$-structure given by \cite[Theorem 5.3]{BLMNPS} applied to the
embedding $c\hookrightarrow C$.

\item An object $\mathcal{F}\in\D^b(\mathcal{X})$ is called $C$-torsion if
it is the pushforward of an object in $\D^b(\mathcal{X}_W)$ for some
proper closed subscheme $W\subset C$.

\item $\mathcal{E}$ is called $C$-torsion free if it
contains no nonzero $C$-torsion subobject.
\end{enumerate}
We denote by $\mathcal{A}_{C\text{-tor}}$ the subcategory of
$C$-torsion objects in $\mathcal{A}_C$, and by
$\mathcal{A}_{C\text{-tf}}$ the subcategory of $C$-torsion free
objects. We say $\mathcal{A}_C$ has a $C$-torsion theory if the pair
of subcategories $(\mathcal{A}_{C\text{-tor}},
\mathcal{A}_{C\text{-tf}})$ forms a torsion pair in the sense of
\cite[Definition 4.6]{BLMNPS}.
\end{definition}

\begin{lemma}\label{flat}
Let $\mathcal{E}\in \mathcal{A}_C$ be as in Definition \ref{tor}.
Then
\begin{enumerate}
\item $\mathcal{E}$ is $C$-flat if and only if $\mathcal{E}$ is
$C$-torsion free;
\item $\mathcal{E}$ is $C$-torsion if and only if
$\mathcal{E}_{K(C)}=0$.
\end{enumerate}
\end{lemma}
\begin{proof}
See \cite[Lemma 6.12 and Lemma 6.4]{BLMNPS}.
\end{proof}

Since  $\Coh(\mathcal{X})$ is the heart of the natural $C$-local
$t$-structure on $\D^b(\mathcal{X})$, one can applies the above
definition and lemma to coherent sheaves. The following lemma shows
the relation of the relative slope-stability to the slop-stability.

\begin{lemma}\label{f-stable}
Let $\mathcal{E}$ be a $C$-torsion free sheaf on $\mathcal{X}$. Then
$\mathcal{E}$ is $\mu_{H, F}$-(semi)stable if and only if there
exists an open subset $U\subset C$ such that $\mathcal{E}_{s}$ is
$\mu_{H_s}$-(semi)stable for any point $s\in U$.
\end{lemma}
\begin{proof}
By Lemma \ref{flat}, one deduces any subsheaf $\mathcal{F}$ of
$\mathcal{E}$ is flat over $C$. Thus
$\mathcal{F}_{K(C)}\in\Coh(\mathcal{X}_{K(C)})$. Since
$$\mu_{H,F}(\mathcal{F})=\mu_{H_{K(C)}}(\mathcal{F}_{K(C)}),$$ from \cite[Ex.
II.5.15]{Hart}, one sees that $\mathcal{E}$ is $\mu_{H,
F}$-(semi)stable if and only if $\mathcal{E}_{K(C)}$ is
$\mu_{H_{K(C)}}$-(semi)stable. Hence the desired conclusion follows
from the openness of slope-stability.
\end{proof}

Another important notion of stability for a sheaf on a fibration is
the stability introduced in \cite[Example 15.3]{BLMNPS}. We define
the slope $\mu_{C}$ of a coherent sheaf $\mathcal{E}\in
\Coh(\mathcal{X})$ by

\begin{eqnarray*}
\mu_{C}(\mathcal{E})= \left\{
\begin{array}{ll}
\frac{FH^{n-2}\ch_1(\mathcal{E})}{FH^{n-1}\ch_0(\mathcal{E})}, &\mbox{if}~\ch_0(\mathcal{E})\neq0,\\
\frac{H^{n-2}\ch_2(\mathcal{E})}{H^{n-1}\ch_1(\mathcal{E})},&\mbox{if}~\mathcal{E}_{K(C)}=0~\mbox{and}~H^{n-1}\ch_1(\mathcal{E})\neq0,\\
+\infty, &\mbox{otherwise}.
\end{array}\right.
\end{eqnarray*}
We can define $\mu_C$-stability as in Definition \ref{slope}.

\begin{definition}
A coherent sheaf $\mathcal{E}$ on $\mathcal{X}$ is
$\mu_{C}$-(semi)stable if, for all non-zero subsheaves
$\mathcal{F}\hookrightarrow \mathcal{E}$, we have
$$\mu_{C}(\mathcal{F})<(\leq)\mu_{C}(\mathcal{E}/\mathcal{F}).$$
\end{definition}

For $\mathcal{E}\in\D^b(\mathcal{X})$, we define
$$Z_{K(C)}(\mathcal{E}):=-FH^{n-2}\ch_1(\mathcal{E})+iFH^{n-1}\ch_0(\mathcal{E})$$ and
$$Z_{C\text{-tor}}(\mathcal{E}):=-H^{n-2}\ch_2(\mathcal{E})+iH^{n-1}\ch_1(\mathcal{E}),$$ respectively.
Then by Lemma \ref{Chern}, one sees that
$$Z_{K(C)}(\mathcal{E})=Z_{C\text{-tor}}(i_{W*}\mathcal{E}_W),$$ for all closed subscheme $W\subset
C$. It follows that $(Z_{K(C)}, Z_{C\text{-tor}})$ is a central
charge on $D^b(\mathcal{X})$ over $C$ in the sense of
\cite[Definition 13.1]{BLMNPS}. By \cite[Proposition 16.6]{BLMNPS},
one deduces that $(Z_{K(C)}, Z_{C\text{-tor}}, \Coh(\mathcal{X}))$
is a weak Harder-Narasimhan structure on $\D^b(\mathcal{X})$ over
$C$ in the sense of \cite[Proposition 15.9]{BLMNPS}. Given a
coherent sheaf $\mathcal{E}$ on $\mathcal{X}$, we let
\begin{eqnarray*}
Z_C(\mathcal{E})= \left\{
\begin{array}{ll}
Z_{K(C)}(\mathcal{E}), &\mbox{if}~\mathcal{E}_{K(C)}\neq0,\\
Z_{C\text{-tor}}(\mathcal{E}),&\mbox{otherwise}.
\end{array}\right.
\end{eqnarray*}
It turns out that $$\mu_{C}(\mathcal{E})=-\frac{\Re
Z_C(\mathcal{E})}{\Im Z_C(\mathcal{E})}.$$ By the definition of the
weak Harder-Narasimhan structure, one sees that like the
slope-stability and the relative slope-stability, the
$\mu_{C}$-stability also satisfies the following weak see-saw
property and the Harder-Narasimhan property.
\begin{proposition}\label{pro2.6}
Let $\mathcal{E}\in\Coh(\mathcal{X})$ be a non-zero sheaf.
\begin{enumerate}
\item For any short exact sequence $$0\rightarrow \mathcal{F}\rightarrow \mathcal{E}\rightarrow
\mathcal{G}\rightarrow0$$ in $\Coh(\mathcal{X})$, we have
$$\mu_{C}(\mathcal{F})\leq\mu_{C}(\mathcal{E})\leq\mu_{C}(\mathcal{G})~\mbox{or}~\mu_{C}(\mathcal{F})\geq\mu_{C}(\mathcal{E})\geq\mu_{C}(\mathcal{G}).$$

\item There is a filtration (called Harder-Narasimhan filtration)
$$0=\mathcal{E}_0\subset \mathcal{E}_1\subset\cdots\subset \mathcal{E}_m=\mathcal{E}$$
such that: $\mathcal{G}_i:=\mathcal{E}_i/\mathcal{E}_{i-1}$ is
$\mu_{C}$-semistable, and
$\mu_{C}(\mathcal{G}_1)>\cdots>\mu_{C}(\mathcal{G}_m)$. We write
$\mu^+_{C}(\mathcal{E}):=\mu_{C}(\mathcal{G}_1)$ and
$\mu^-_{C}(\mathcal{E}):=\mu_{C}(\mathcal{G}_m)$.
\end{enumerate}
\end{proposition}

The below lemma gives the relation between the Chern characters of
objects on fibers and their pushforwards.

\begin{lemma}\label{Chern}
Let $W$ be a closed subscheme of $C$, and $j$ be a positive integer.
\begin{enumerate}
\item For any $\mathcal{E}\in\D^b(\mathcal{X})$ and
$\mathbb{Q}$-divisor $D$ on $\mathcal{X}$, we have
$$\ch_j^D(i_{W*}\mathcal{E}_W)=\mathcal{X}_W\ch^D_{j-1}(\mathcal{E}).$$
\item Assume that $W$ is a closed point of $C$. Then for any
$\mathcal{Q}\in\D^b(\mathcal{X}_W)$ we have
$$\ch_{j}(i_{W*}\mathcal{Q})=i_{W*}\ch_{j-1}(\mathcal{Q}).$$
\end{enumerate}
\end{lemma}
\begin{proof}
(1) Since $\mathcal{X}_W$ is a divisor of $\mathcal{X}$, from the
standard exact triangle
$$\mathcal{E}\otimes\mathcal{O}_{\mathcal{X}}(-\mathcal{X}_W)\rightarrow \mathcal{E}\rightarrow i_{W*}\mathcal{E}_W,$$
one sees
\begin{eqnarray*}
\ch_j^D(i_{W*}\mathcal{E}_W)&=&\ch_j^D(\mathcal{E})-\ch_j^D(\mathcal{E}\otimes\mathcal{O}_{\mathcal{X}}(-\mathcal{X}_W))\\
&=&\ch_j^D(\mathcal{E})-\left(\ch_j^D(\mathcal{E})-\mathcal{X}_W\ch^D_{j-1}(\mathcal{E})+\frac{1}{2}\mathcal{X}^2_W\ch^D_{j-2}(\mathcal{E})+\cdots\right)\\
&=&\mathcal{X}_W\ch^D_{j-1}(\mathcal{E}).
\end{eqnarray*}
(2) Applying the Grothendieck-Riemann-Roch theorem for the embedding
$i_W:\mathcal{X}_W\hookrightarrow \mathcal{X}$, we conclude that
$$\ch(i_{W*}\mathcal{Q})=i_{W*}
\Big(\ch(\mathcal{Q})(td(\mathcal{O}_{\mathcal{X}_W}))^{-1}\Big)=i_{W*}
\ch(\mathcal{Q}).$$ This implies the desired equalities.
\end{proof}

Lemma \ref{f-stable} says that the usual notion of relative
slope-stability for a torsion free sheaf is equivalent to the
slope-stability of the general fiber of the sheaf. In contrast,
$\mu_C$-stability requires stability for all fibers:
\begin{proposition}\label{pro2.8}
Let $\mathcal{E}$ be a $C$-torsion free sheaf on $\mathcal{X}$. Then
$\mathcal{E}$ is $\mu_C$-semistable if and only if $\mathcal{E}$ is
$\mu_{H,F}$-semistable and for any closed point $p\in C$ and any
quotient $\mathcal{E}_p\twoheadrightarrow \mathcal{Q}$ in
$\Coh(\mathcal{X}_p)$ we have
$\mu_{H_p}(\mathcal{E}_p)\leq\mu_{H_p}(\mathcal{Q})$.
\end{proposition}
\begin{proof}
See \cite[Lemma 15.7]{BLMNPS}.
\end{proof}

\section{Relative tilt-stability}\label{S3}
In this section, we recall the definition of very weak stability
conditions on $\mathcal{D}$ introduced in \cite[Appendix 2]{BMS},
\cite[Section 2.1]{PT} and \cite[Section 2]{Toda1} and give a relative version of the
tilt-stability constructed in \cite{BMT}. We keep the same notations
as that in the previous sections.

\subsection{Very weak stability condition}
Let $\mathcal{D}$ be a triangulated category, for which we fix a
finitely generated free abelian group $\Lambda$ and a group
homomorphism $v:\K(\mathcal{D})\rightarrow\Lambda$.
\begin{definition}\label{pre}
A very weak stability condition on $\mathcal{D}$ is a pair
$\sigma=(Z, \mathcal{A})$, where $\mathcal{A}$ is the heart of a
bounded t-structure on $\mathcal{D}$, and $Z:\Lambda\rightarrow
\mathbb{C}$ is a group homomorphism (called central charge) such
that
\begin{enumerate}
\item $Z$ satisfies the following positivity property for any $\mathcal{E}\in
\mathcal{A}$:
$$Z(v(\mathcal{E}))\in\{re^{i\pi\phi}: r\geq0, 0<\phi\leq1\}.$$
\item $(Z, \mathcal{A})$ satisfies the Harder-Narasimhan
property: every object of $\mathcal{A}$ has a Harder-Narasimhan
filtration in $\mathcal{A}$ with respect to
$\nu_{\sigma}$-stability, here the slope $\nu_{\sigma}$ of an object
$\mathcal{E}\in \mathcal{A}$ is defined by
\begin{eqnarray*}
\nu_{\sigma}(\mathcal{E})= \left\{
\begin{array}{lcl}
+\infty,  & &\mbox{if}~\Im Z(v(\mathcal{E}))=0,\\
&&\\
-\frac{\Re Z(v(\mathcal{E}))}{\Im Z(v(\mathcal{E}))}, &
&\mbox{otherwise}.
\end{array}\right.
\end{eqnarray*}
\end{enumerate}
\end{definition}

A very weak stability condition $\sigma=(Z, \mathcal{A})$ is called
a stability condition if for any $0\neq \mathcal{E}\in \mathcal{A}$
we have $Z(v(\mathcal{E}))\neq0$. This notion coincides with the
notion of Bridgeland stability conditions \cite{Bri1}.

We say $\mathcal{E}\in\mathcal{A}$ is $\nu_{\sigma}$-(semi)stable if
for any non-zero subobject $\mathcal{F}\subset \mathcal{E}$ in
$\mathcal{A}$, we have
$$\nu_{\sigma}(\mathcal{F})<(\leq)\nu_{\sigma}(\mathcal{E}/\mathcal{F}).$$
The Harder-Narasimhan filtration of an object $\mathcal{E}\in
\mathcal{A}$ is a chain of subobjects
$$0=\mathcal{E}_0\subset \mathcal{E}_1\subset\cdots\subset \mathcal{E}_m=\mathcal{E}$$ in $\mathcal{A}$ such
that $\mathcal{G}_i:=\mathcal{E}_i/\mathcal{E}_{i-1}$ is
$\nu_{\sigma}$-semistable and
$\nu_{\sigma}(\mathcal{G}_1)>\cdots>\nu_{\sigma}(\mathcal{G}_m)$. We
set $\nu_{\sigma}^+(\mathcal{E}):=\nu_{\sigma}(\mathcal{G}_1)$ and
$\nu_{\sigma}^-(\mathcal{E}):=\nu_{\sigma}(\mathcal{G}_m)$.

\begin{definition}\label{slicing}
For a very weak stability condition $(Z, \mathcal{A})$ on
$\mathcal{D}$ and for $0<\phi\leq1$, we define the subcategory
$\mathcal{P}(\phi)\subset\mathcal{D}$ to be the category of
$\nu_{\sigma}$-semistable objects $\mathcal{E}\in\mathcal{A}$
satisfying $\tan(\pi\phi)=-1/\nu_{\sigma}(\mathcal{E})$. For other
$\phi\in \mathbb{R}$ the subcategory $\mathcal{P}(\phi)$ is defined
by the rule:
$$\mathcal{P}(\phi+1)=\mathcal{P}(\phi)[1].$$ The objects in $\mathcal{P}(\phi)$ is still called $\nu_{\sigma}$-semistable
objects.
\end{definition}

For an interval $I=(a,b)\subset\mathbb{R}$, we denote by
$\mathcal{P}(I)$ the extension-closure of
$$\bigcup_{\phi\in I}\mathcal{P}(\phi)\subset\mathcal{D}.$$ $\mathcal{P}(I)$ is a
quasi-abelian category when $b-a<1$ (cf. \cite[Definition
4.1]{Bri1}). If we have a distinguished triangle
$$A_1\xrightarrow{h} A_2\xrightarrow{g} A_3\rightarrow A_1[1]$$ with
$A_1,A_2,A_3\in \mathcal{P}(I)$, we say $h$ is a strict monomorphism
and $g$ is a strict epimorphism. Then we say that $\mathcal{P}(I)$
is of finite length if $\mathcal{P}(I)$ is Noetherian and Artinian
with respect to strict epimorphisms and strict monomorphisms,
respectively.

\begin{definition}
A very weak stability condition $\sigma=(Z,\mathcal{A})$ is called
locally finite if there exists $\varepsilon>0$ such that for any
$\phi\in \mathbb{R}$, the quasi-abelian category
$\mathcal{P}((\phi-\varepsilon, \phi+\varepsilon))$ is of finite
length.
\end{definition}

In Definition \ref{pre}, we let $\Lambda_0$ be the saturation of the
subgroup of $\Lambda$ generated by
$$\{v(\mathcal{E}):\mathcal{E}\in \mathcal{A}, Z(v(\mathcal{E}))=0\}.$$
Note that $Z$ descends to the group homomorphism
$Z^{\prime}:\Lambda/\Lambda_0\rightarrow\mathbb{C}$. For
$w\in\Lambda$, we denote by $w^{\prime}$ its image in
$\Lambda/\Lambda_0$, and let $\|*\|$ be a fixed norm on
$(\Lambda/\Lambda_0)\otimes_{\mathbb{Z}}\mathbb{R}$.

\begin{definition}\label{supp}
We say a very weak stability condition $\sigma=(Z,\mathcal{A})$
satisfies the support property if there is a quadratic form $Q$ on
$\Lambda/\Lambda_0$ satisfying $Q(v(\mathcal{E}))\geq0$ for any
$\nu_{\sigma}$-semistable object $\mathcal{E}\in\mathcal{A}$, and
$Q|_{\ker Z^{\prime}}$ is negative definite.
\end{definition}

\begin{remark}\label{remark3.5}
The local finiteness condition automatically follows if the support
property is satisfied (cf. \cite[Section 1.2]{KS} and \cite[Lemma
4.5]{Bri2}).
\end{remark}

\subsection{Relative tilt-stability} Let $\beta$ be a rational
number and $\alpha$ be a positive real number such that
$\alpha^2\in\mathbb{Q}$. We will construct a family of very weak
stability conditions on $\D^b(\mathcal{X})$ that depends on these
two parameters. For brevity, we write $\ch^{\beta}$ for the twisted
Chern character $\ch^{\beta H}$.

There exists a \emph{torsion pair} $(\mathcal{T}_{\beta
H},\mathcal{F}_{\beta H})$ in $\Coh(\mathcal{X})$ defined as
follows:
\begin{eqnarray*}
\mathcal{T}_{\beta H}&=&\{\mathcal{E}\in\Coh(\mathcal{X}):\mu^-_{C}(\mathcal{E})>\beta \}\\
\mathcal{F}_{\beta
H}&=&\{\mathcal{E}\in\Coh(\mathcal{X}):\mu^+_{C}(\mathcal{E})\leq\beta
\}.
\end{eqnarray*}
Equivalently, $\mathcal{T}_{\beta H}$ and $\mathcal{F}_{\beta H}$
are the extension-closed subcategories of $\Coh(\mathcal{X})$
generated by $\mu_{C}$-stable sheaves with $\mu_{C}$-slope $>\beta$
and $\leq\beta$, respectively.

\begin{definition}\label{def3.3}
We let $\Coh_C^{\beta H}(\mathcal{X})\subset \D^b(\mathcal{X})$ be
the extension-closure
$$\Coh_C^{\beta H}(\mathcal{X})=\langle\mathcal{T}_{\beta H}, \mathcal{F}_{
\beta H}[1]\rangle.$$
\end{definition}

By the general theory of torsion pairs and tilting \cite{HRS},
$\Coh_C^{\beta H}(\mathcal{X})$ is the heart of a bounded
t-structure on $\D^b(\mathcal{X})$; in particular, it is an abelian
category. For any point $s\in C$, similar as Definition
\ref{def3.3}, one can define
the subcategory
$$\Coh^{\beta H_s}(\mathcal{X}_s)=\langle\mathcal{T}_{\beta H_s}, \mathcal{F}_{
\beta H_s}[1]\rangle\subset\D^b(\mathcal{X}_s)$$  via the
$\mu_{H_s}$-stability (see \cite[Section 14.2]{BLMNPS}).

Consider the following central charge
$$z_{H,F}^{\alpha, \beta}(\mathcal{E})=\frac{\alpha^2 }{2}FH^{n-1}\ch_0^{\beta}(\mathcal{E})-FH^{n-3}\ch_2^{\beta}(\mathcal{E})+i FH^{n-2}\ch_1^{\beta}(\mathcal{E})
,$$ here we set $FH^{n-3}=1$ if $n=2$. We think of it as the
composition
$$z_{H,F}^{\alpha, \beta}: \K(\D^b(\mathcal{X}))\xrightarrow{v} \mathbb{Z}\oplus\mathbb{Z}\oplus\frac{1}{2}\mathbb{Z} \xrightarrow{Z_{H,F}^{\alpha,\beta}}
\mathbb{C},$$ where the first map is given by
$$v(\mathcal{E})=(FH^{n-1}\ch_0(\mathcal{E}), FH^{n-2}\ch_1(\mathcal{E}), FH^{n-3}\ch_2(\mathcal{E})),$$
and the second map is defined by
\begin{equation*}
Z_{H,F}^{\alpha, \beta}(e_0, e_1,
e_2)=\frac{1}{2}(\alpha^2-\beta^2)e_0+\beta e_1-e_2+i(e_1-\beta
e_0).
\end{equation*}

We recall the classical Bogomolov inequality:
\begin{theorem}\label{Bog}
Assume that $\mathcal{E}$ is a $\mu_{H,F}$-semistable torsion free
sheaf on $\mathcal{X}$. Then we have
\begin{eqnarray*}
FH^{n-3}\Delta(\mathcal{E})&:=&FH^{n-3}\big(\ch_1^2(\mathcal{E})-2\ch_0(\mathcal{E})\ch_2(\mathcal{E})\big)\geq0;\\
H^{n-2}\Delta(\mathcal{E})&:=&H^{n-2}\big(\ch_1^2(\mathcal{E})-2\ch_0(\mathcal{E})\ch_2(\mathcal{E})\big)\geq0.
\end{eqnarray*}
\end{theorem}
\begin{proof}
See \cite[Theorem 3.2]{Langer1}.
\end{proof}

A short calculation shows
\begin{eqnarray*}
\Delta(\mathcal{E})&:=&(\ch_1(\mathcal{E}))^2-2\ch_0(\mathcal{E})\ch_2(\mathcal{E})\\
&=&(\ch^{\beta}_1(\mathcal{E}))^2-2\ch^{\beta}_0(\mathcal{E})\ch^{\beta}_2(\mathcal{E}).
\end{eqnarray*}

\begin{definition}
We define the generalized relative discriminants
$$\overline{\Delta}^{\beta
H}_{H,F}:=(FH^{n-2}\ch^{\beta}_1)^2-2FH^{n-1}\ch^{\beta}_0\cdot(FH^{n-3}\ch^{\beta}_2)$$
and $$\widetilde{\Delta}^{\beta
H}_{H,F}:=(FH^{n-2}\ch_1^{\beta})(H^{n-1}\ch_1^{\beta})-FH^{n-1}\ch^{\beta}_0\cdot(H^{n-2}\ch^{\beta}_2).$$
\end{definition}
A short calculation shows $$\overline{\Delta}^{\beta
H}_{H,F}=(FH^{n-2}\ch_1)^2-2FH^{n-1}\ch_0\cdot(FH^{n-3}\ch_2)=\overline{\Delta}_{H,F}$$
when $n\geq3$. Hence the first generalized relative discriminant
$\overline{\Delta}^{\beta H}_{H,F}$ is independent of $\beta$ when
$n\geq3$. In general $\widetilde{\Delta}^{\beta H}_{H,F}$ is not
independent of $\beta$, but we have $$\widetilde{\Delta}^{\beta
H}_{H,F}(\mathcal{E}\otimes\mathcal{O}_{\mathcal{X}}(mF))=\widetilde{\Delta}^{\beta
H}_{H,F}(\mathcal{E}),$$ for any $\mathcal{E}\in\D^b(\mathcal{X})$
and $m\in\mathbb{Z}$.

\begin{lemma}\label{Hodge}
Let $D$ be a $\mathbb{Q}$-divisor on $\mathcal{X}$. Then we have
\begin{enumerate}
\item $(H^{n-1}F)(FH^{n-3}D^2)\leq(H^{n-2}FD)^2$ if $n\geq3$;
\item $(H^{n-1}F)(H^{n-2}D^2)\leq2(H^{n-1}D)(H^{n-2}FD)$ if $n\geq2$.
\end{enumerate}
\end{lemma}
\begin{proof}
Since $$FH^{n-2}\Big((FH^{n-1})D-(FDH^{n-2})H\Big)=0,$$ the Hodge
index theorem gives
$$FH^{n-3}\Big((FH^{n-1})D-(FDH^{n-2})H\Big)^2\leq0.$$ An easy
computation shows the inequality (1) holds.

For the inequality (2), one notices that
$$H^{n-1}\Big((H^{n-1}F)D-(H^{n-1}D)F\Big)=0.$$ From the Hodge index
theorem, it follows that
$$H^{n-2}\Big((H^{n-1}F)D-(H^{n-1}D)F\Big)^2\leq0.$$ Expanding the
left hand side of the above inequality, one obtains the desired
conclusion.
\end{proof}

By Lemma \ref{Hodge} and Theorem \ref{Bog}, we have:
\begin{theorem}\label{Bog1}
Assume that $\mathcal{E}$ is a $\mu_{H,F}$-semistable torsion free
sheaf on $\mathcal{X}$. Then we have $\overline{\Delta}^{\beta
H}_{H,F}(\mathcal{E})\geq0$ when $n\geq3$ and
$\widetilde{\Delta}^{\beta H}_{H,F}(\mathcal{E})\geq0$ when
$n\geq2$.
\end{theorem}

The following theorem gives a relative version of the tilt-stability
in \cite{BMT}.
\begin{theorem}\label{tilt}
For any $(\alpha,\beta)\in \sqrt{\mathbb{Q}_{>0}}\times\mathbb{Q}$,
$\sigma_{H,F}^{\alpha, \beta}=(Z_{H,F}^{\alpha, \beta},
\Coh_C^{\beta H}(\mathcal{X}))$ is a very weak stability condition.
\end{theorem}
\begin{proof}
\textbf{Step 1.} The pair $\sigma_{H,F}^{\alpha,
\beta}=(Z_{H,F}^{\alpha, \beta}, \Coh_C^{\beta H}(\mathcal{X}))$
satisfies the positivity property for any $0\neq\mathcal{E}\in
\Coh_C^{\beta H}(\mathcal{X})$.
\bigskip

By the construction of $\Coh_C^{\beta H}(\mathcal{X})$, one sees
that
$$\Im Z_{H,F}^{\alpha,\beta}(v(\mathcal{E}))=FH^{n-2}\ch^{\beta}_1(\mathcal{E})=FH^{n-2}\ch_1(\mathcal{E})-\beta
FH^{n-1}\ch_0(\mathcal{E})\geq0,$$ for any
$\mathcal{E}\in\Coh_C^{\beta H}(\mathcal{X})$. Now we assume that
$FH^{n-2}\ch^{\beta}_1(\mathcal{E})=0$. One obtains
\begin{equation}\label{2.1}
\ch_0(\mathcal{H}^0(\mathcal{E}))=FH^{n-2}\ch_1(\mathcal{H}^0(\mathcal{E}))=0,
~\mu^-_C(\mathcal{H}^0(\mathcal{E}))>\beta
\end{equation}
and one of the following cases occurs:
\begin{enumerate}
\item $\ch_0(\mathcal{H}^{-1}(\mathcal{E}))>0$ and
$\mathcal{H}^{-1}(\mathcal{E})$ is $\mu_C$-semistable with
$\mu_C(\mathcal{H}^{-1}(\mathcal{E}))=\beta$;

\item
$\mathcal{H}^{-1}(\mathcal{E})$ is $C$-torsion and
$\mu^+_C(\mathcal{H}^{-1}(\mathcal{E}))\leq\beta$.
\end{enumerate}

One sets $\mathcal{G}$ be the maximal subsheaf of
$\mathcal{H}^0(\mathcal{E})$ whose support has codimension $\geq2$.
Then $\mathcal{H}^0(\mathcal{E})/\mathcal{G}$ is pure and
$C$-torsion. The condition (\ref{2.1}) implies that
$$H^{n-1}\ch_1(\mathcal{H}^{0}(\mathcal{E}))=H^{n-1}\ch_1(\mathcal{H}^{0}(\mathcal{E})/\mathcal{G})\geq0$$
and

$$\mu_C(\mathcal{H}^{0}(\mathcal{E})/\mathcal{G})=\frac{H^{n-2}\ch_2(\mathcal{H}^{0}(\mathcal{E})/\mathcal{G})}{H^{n-1}
\ch_1(\mathcal{H}^{0}(\mathcal{E})/\mathcal{G})}=
\frac{H^{n-2}\ch^{\beta}_2(\mathcal{H}^{0}(\mathcal{E})/\mathcal{G})}{H^{n-1}\ch_1(\mathcal{H}^{0}(\mathcal{E})/\mathcal{G})}+\beta>\beta.$$
Thus we get
$H^{n-2}\ch^{\beta}_2(\mathcal{H}^0(\mathcal{E}))=H^{n-2}\ch^{\beta}_2(\mathcal{H}^0(\mathcal{E})/\mathcal{G})+H^{n-2}\ch^{\beta}_2(\mathcal{G})\geq0$.
When $n\geq3$, one has
$FH^{n-3}\ch^{\beta}_2(\mathcal{H}^0(\mathcal{E})/\mathcal{G})=0$
and $FH^{n-3}\ch^{\beta}_2(\mathcal{G})\geq0$. Hence we obtain
$$FH^{n-3}\ch^{\beta}_2(\mathcal{H}^0(\mathcal{E}))\geq0$$
when $n\geq2$.

On the other hand, in the case (2), since
$\mu^+_C(\mathcal{H}^{-1}(\mathcal{E}))\leq\beta$, one sees that
$H^{n-1}\ch_1(\mathcal{H}^{-1}(\mathcal{E}))>0$ and
$$\mu_C(\mathcal{H}^{-1}(\mathcal{E}))=\frac{H^{n-2}\ch_2(\mathcal{H}^{-1}(\mathcal{E}))}{H^{n-1}\ch_1(\mathcal{H}^{-1}(\mathcal{E}))}=
\frac{H^{n-2}\ch^{\beta}_2(\mathcal{H}^{-1}(\mathcal{E}))}{H^{n-1}\ch_1(\mathcal{H}^{-1}(\mathcal{E}))}+\beta\leq\beta.$$
These imply that
$H^{n-2}\ch^{\beta}_2(\mathcal{H}^{-1}(\mathcal{E}))\leq0$ when
$n\geq2$ and
$FH^{n-3}\ch^{\beta}_2(\mathcal{H}^{-1}(\mathcal{E}))=0$ when
$n\geq3$. Therefore in the case (2) we obtain $$\Re
Z_{H,F}^{\alpha,\beta}(v(\mathcal{E}))=\frac{\alpha^2}{2}FH^{n-1}\ch_0^{\beta}(\mathcal{E})-FH^{n-3}\ch_2^{\beta}(\mathcal{E})\leq0.$$
For the case (1), we let $\mathcal{T}$ be the torsion part of
$\mathcal{H}^{-1}(\mathcal{E})$. One sees that $\mathcal{T}$ is
$C$-torsion and
$$FH^{n-2}\ch^{\beta}_1(\mathcal{T})=FH^{n-2}\ch^{\beta}_1(\mathcal{H}^{-1}(\mathcal{E})/\mathcal{T})=0~\mbox{and} ~\mu^+_C(\mathcal{T})\leq\beta.$$
By the same way as the proof in the case (2), one obtains
$H^{n-2}\ch^{\beta}_2(\mathcal{H}^{-1}(\mathcal{T}))\leq0$ when
$n\geq2$ and
$FH^{n-3}\ch^{\beta}_2(\mathcal{H}^{-1}(\mathcal{T}))=0$ when
$n\geq3$. Thus $FH^{n-3}\ch_2^{\beta}(\mathcal{T})\leq0$. Since
$FH^{n-2}\ch^{\beta}_1(\mathcal{H}^{-1}(\mathcal{E})/\mathcal{T})=0$,
we infer that
$$\mu_{H,F}(\mathcal{H}^{-1}(\mathcal{E}))=\mu_{H,F}(\mathcal{H}^{-1}(\mathcal{E})/\mathcal{T})=\beta.$$
Thus $\mathcal{H}^{-1}(\mathcal{E})/\mathcal{T}$ is a
$\mu_{H,F}$-semistable sheaf by Proposition \ref{pro2.8}. By Theorem
\ref{Bog1} we have
$H^{n-2}\ch^{\beta}_2(\mathcal{H}^{-1}(\mathcal{E})/\mathcal{T})\leq0$
if $n\geq2$ and
$FH^{n-3}\ch^{\beta}_2(\mathcal{H}^{-1}(\mathcal{E})/\mathcal{T})\leq0$
if $n\geq3$. Hence
$FH^{n-3}\ch^{\beta}_2(\mathcal{H}^{-1}(\mathcal{E}))\leq0$, and in
the case (1) we conclude $\Re
Z_{H,F}^{\alpha,\beta}(v(\mathcal{E}))<0$.

\bigskip
\textbf{Step 2.} The category $\Coh_C^{\beta H}(\mathcal{X})$ is
noetherian, and $\sigma_{H,F}^{\alpha, \beta}$ satisfies the
Harder-Narasimhan property.
\bigskip

For every $s\in C$ we let
$$\sigma^{\sharp\beta}_s:=\left(Z_s=\ch^{\beta}_{\mathcal{X}_s,0}+i\ch^{\beta}_{\mathcal{X}_s,1}, \Coh^{\beta
H_s}(\mathcal{X}_s)\right).$$ By \cite[Proposition 25.1]{BLMNPS},
one sees that the collection
$\underline{\sigma}^{\sharp\beta}:=(\sigma^{\sharp\beta}_s)$ is a
flat family of fiberwise weak stability condition on
$\D^b(\mathcal{X})$ over $C$ (See \cite[Definition 20.5]{BLMNPS}).
Hence from \cite[Corollary 20.10]{BLMNPS} and \cite[Proposition
15.14]{BLMNPS}, it follows that $\Coh_C^{\beta H}(\mathcal{X})$ has
a $C$-torsion theory and is noetherian. Since
$Z_{H,F}^{\alpha,\beta}$ has discrete image, we conclude that the
Harder-Narasimhan filtrations exist for objects in $\Coh_C^{\beta
H}(\mathcal{X})$ with respect to $Z_{H,F}^{\alpha, \beta}$ (cf.
\cite[Lemma 2.18]{PT}).
\end{proof}

\begin{remark}\label{rem}
Let $\mathcal{E}$ be an object in $\Coh_C^{\beta H}(\mathcal{X})$
with $Z_{H,F}^{\alpha, \beta}(v(\mathcal{E}))=0$. Since
$\Coh_C^{\beta H}(\mathcal{X})$ has a $C$-torsion theory, we denote
by $\mathcal{E}_{C\text{-tor}}$ and $\mathcal{E}_{C\text{-tf}}$ the
$C$-torsion part and $C$-torsion free part of $\mathcal{E}$,
respectively. Then one sees that
$$Z_{H,F}^{\alpha,
\beta}(v(\mathcal{E}_{C\text{-tor}}))=Z_{H,F}^{\alpha,
\beta}(v(\mathcal{E}_{C\text{-tf}}))=0.$$ By the proof of Theorem
\ref{tilt}, we can deduce that
$\mathcal{H}^{-1}(\mathcal{E}_{C\text{-tf}})=0$ and
$\mathcal{H}^{0}(\mathcal{E}_{C\text{-tf}})\in\Coh_{\leq
n-3}(\mathcal{X})$. In particular, $\mathcal{E}_{C\text{-tf}}=0$
when $n\leq3$. On the other hand, one has $Z_{H,F}^{\alpha,
\beta}(v(\mathcal{F}))=0$ for any $C$-torsion object
$\mathcal{F}\in\Coh_C^{\beta H}(\mathcal{X})$ when $n\geq3$. Hence
we conclude that if $n=3$, then $\mathcal{F}\in\Coh_C^{\beta
H}(\mathcal{X})$ is $C$-torsion is equivalent to $Z_{H,F}^{\alpha,
\beta}(v(\mathcal{F}))=0$.
\end{remark}

\begin{remark}
The pair $\sigma_F^{\alpha, \beta}=(Z_F^{\alpha, \beta},
\Coh_C^{\beta H}(\mathcal{X}))$ in Theorem \ref{tilt} is not a
stability condition, since $$Z_F^{\alpha,
\beta}(\mathcal{O}_F)=Z_F^{\alpha, \beta}(\mathcal{O}_F[1])=0$$ and
one of $\mathcal{O}_F$ and $\mathcal{O}_F[1]$ is in $\Coh_C^{\beta
H}(\mathcal{X})$.
\end{remark}

\begin{lemma}\label{ch2}
Let $\mathcal{E}$ be an object in $\Coh_C^{\beta H}(\mathcal{X})$.
\begin{enumerate}
\item We have $FH^{n-2}\ch^{\beta}_1(\mathcal{E})\geq0$.
\item If $FH^{n-2}\ch^{\beta}_1(\mathcal{E})=0$, then one has
$H^{n-2}\ch^{\beta}_2(\mathcal{E})\geq0$,
$FH^{n-3}\ch^{\beta}_2(\mathcal{E})\geq0$ and
$\ch_0(\mathcal{E})\leq0$.
\item If $FH^{n-2}\ch^{\beta}_1(\mathcal{E})=\ch_0(\mathcal{E})=H^{n-2}\ch^{\beta}_2(\mathcal{E})=0$, then
$\mathcal{H}^0(\mathcal{E})\in\Coh_{\leq n-3}(\mathcal{X})$,
$\mathcal{H}^{-1}(\mathcal{E})$ is a $C$-torsion $\mu_C$-semistable
sheaf with $\mu_C(\mathcal{H}^{-1}(\mathcal{E}))=\beta$ and
$H^{n-3}\ch^{\beta}_3(\mathcal{E})\geq0$.
\end{enumerate}
\end{lemma}
\begin{proof}
The first statement follows from the definition of $\Coh_C^{\beta
H}(\mathcal{X})$. By Step 1 in the proof of Theorem \ref{tilt}, one
obtains the second statement.

For the third statement, still by Step 1 in the proof of Theorem
\ref{tilt}, one sees that if $FH^{n-2}\ch^{\beta}_1(\mathcal{E})=0$,
$\ch_0(\mathcal{E})=0$ and $H^{n-2}\ch^{\beta}_2(\mathcal{E})=0$
then the support of $\mathcal{H}^0(\mathcal{E})$ has codimension
$\geq3$ and $\mathcal{H}^{-1}(\mathcal{E})$ is a $C$-torsion
$\mu_C$-semistable sheaf with
$\mu_C(\mathcal{H}^{-1}(\mathcal{E}))=\beta$. Hence under the
assumptions of the third statement we have
$H^{n-3}\ch^{\beta}_3(\mathcal{H}^{0}(\mathcal{E}))\geq0$. Now we
prove that
$H^{n-3}\ch^{\beta}_3(\mathcal{H}^{-1}(\mathcal{E}))\leq0$. Without
loss of generality, we can assume that
$\mathcal{H}^{-1}(\mathcal{E})$ is a $C$-torsion sheaf
set-theoretically supported over a closed point $s\in C$. Let $\pi$
be a local generator of $I_s$. By \cite[Lemma 6.11]{BLMNPS}, one
obtains a filtration
$$0=\mathcal{G}_m\subset\mathcal{G}_{m-1}\subset\cdots\subset\mathcal{G}_1\subset\mathcal{G}_0=\mathcal{H}^{-1}(\mathcal{E})$$
where $\mathcal{G}_j=\pi^j\cdot\mathcal{H}^{-1}(\mathcal{E})$ and
all filtration quotients $\mathcal{G}_{j}/\mathcal{G}_{j+1}$ are
quotients of $\mathcal{G}_{0}/\mathcal{G}_{1}$ in
$i_{s*}(\Coh(\mathcal{X}_s))$. Since $\mathcal{H}^{-1}(\mathcal{E})$
is a $\mu_C$-semistable sheaf with
$\mu_C(\mathcal{H}^{-1}(\mathcal{E}))=\beta$, so are $\mathcal{G}_j$
and $\mathcal{G}_{j}/\mathcal{G}_{j+1}$ for $j=0,1,\cdots, m-1$. We
write $\mathcal{G}_{j}/\mathcal{G}_{j+1}=i_{s*}(\mathcal{F}_j)$,
here $\mathcal{F}_j\in\Coh(\mathcal{X}_s)$. Then from Bogomolov's
inequality for the semistable sheaves $\mathcal{F}_j$ on
$\mathcal{X}_s$ and Lemma \ref{Chern}, it follows that
$$H^{n-3}\ch^{\beta}_3(i_{s*}(\mathcal{F}_j))=H^{n-3}i_{s*}(\ch^{\beta}_2(\mathcal{F}_j))=H_s^{n-3}\ch^{\beta}_2(\mathcal{F}_j)\leq0.$$
This implies
$$H^{n-3}\ch^{\beta}_3(\mathcal{H}^{-1}(\mathcal{E}))=\sum_{j=0}^{m-1}H^{n-3}\ch^{\beta}_3(\mathcal{G}_{j}/\mathcal{G}_{j+1})\leq0.$$
Therefore one concludes that
$$H^{n-3}\ch^{\beta}_3(\mathcal{E})=H^{n-3}\ch^{\beta}_3(\mathcal{H}^{0}(\mathcal{E}))+H^{n-3}\ch^{\beta}_3(\mathcal{H}^{-1}(\mathcal{E})[1])\geq0.$$
This proves the third statement.
\end{proof}

We write $\nu_{H,F}^{\alpha, \beta}$ for the slope function on
$\Coh_C^{ \beta H}(\mathcal{X})$ induced by $Z_{H,F}^{\alpha,
\beta}$. Explicitly, for any $\mathcal{E}\in \Coh_C^{ \beta
H}(\mathcal{X})$, one has
\begin{eqnarray*}
\nu_{H,F}^{\alpha, \beta}(\mathcal{E})= \left\{
\begin{array}{lcl}
+\infty,  & &\mbox{if}~FH^{n-2}\ch^{\beta}_1(\mathcal{E})=0,\\
&&\\
\frac{FH^{n-3}\ch_2^{\beta}(\mathcal{E})-\frac{1}{2}\alpha^2FH^{n-1}\ch^{\beta}_0(\mathcal{E})}{FH^{n-2}\ch^{\beta}_1(\mathcal{E})},
& &\mbox{otherwise}.
\end{array}\right.
\end{eqnarray*}
Theorem \ref{tilt} gives the notion of
$\nu_{H,F}^{\alpha,\beta}$-stability. We can also consider the
tilt-stability on the fibers of $f$. If $n\geq3$, for any point
$s\in C$, we define
$$\sigma_s^{\alpha,\beta}:=\left(Z_s^{\alpha,\beta}=iH^{n-2}_s\ch_{\mathcal{X}_s,1}^{\beta}
+\frac{\alpha^2}{2}H_s^{n-1}\ch_{\mathcal{X}_s,0}^{\beta}-H^{n-3}_s\ch_{\mathcal{X}_s,2}^{\beta},
\Coh^{\beta H_s}(\mathcal{X}_s) \right).$$ This is the
tilt-stability condition on $\mathcal{X}_s$ defined in \cite{BMT,
BMS}. We write $\nu_s^{\alpha, \beta}$ for the slope function on
$\Coh^{\beta H_s}(\mathcal{X}_s)$ induced by $Z_s^{\alpha, \beta}$.
One sees
\begin{eqnarray*}
\nu_s^{\alpha,
\beta}(\mathcal{E}_s)=\nu_{H,F}^{\alpha,\beta}(\mathcal{E}).
\end{eqnarray*}
We also call $\nu_{H,F}^{\alpha,\beta}$-stability relative
tilt-stability.

\begin{lemma}\label{K-stable}
Let $\mathcal{E}\in \Coh_C^{ \beta H}(\mathcal{X})$ be a $C$-torsion
free object. Then the following conditions are equivalent:
\begin{enumerate}
\item $\mathcal{E}$ is $\nu_{H,F}^{\alpha, \beta}$-(semi)stable;

\item $\mathcal{E}_{K(C)}$ is $\nu_{K(C)}^{\alpha,
\beta}$-(semi)stable;

\item there exists an open subset $U\subset C$ such that
$\mathcal{E}_{s}$ is $\nu_{s}^{\alpha, \beta}$-(semi)stable for any
point $s\in U$.
\end{enumerate}
\end{lemma}
\begin{proof}
By Lemma \ref{flat}, one infers that any subobject $\mathcal{F}$ of
$\mathcal{E}$ is $C$-flat. Thus $\mathcal{F}_{K(C)}\in\Coh^{\beta
H_{K(C)}}(\mathcal{X}_{K(C)})$. Since
\begin{eqnarray*}
\nu_{K(C)}^{\alpha,
\beta}(\mathcal{F}_{K(C)})=\nu_{H,F}^{\alpha,\beta}(\mathcal{F}),
\end{eqnarray*} from \cite[Lemma 4.16.(2)]{BLMNPS}, one deduces that  $\mathcal{E}$ is $\nu_{H,F}^{\alpha,
\beta}$-(semi)stable if and only if $\mathcal{E}_{K(C)}$ is
$\nu_{K(C)}^{\alpha, \beta}$-(semi)stable. The implication
``$(3)\Rightarrow (2)$'' is obvious. For the other direction, by
\cite[Proposition 25.3]{BLMNPS}, one sees that
$(\sigma_s^{\alpha,\beta})_{s\in C}$ is a flat family of fiberwise
weak stability condition on $\D^b(\mathcal{X})$ over $C$. Hence
\cite[Definition 20.5.(2') and Lemma 20.4]{BLMNPS} gives the
openness of tilt-stability and tilt-semistability. This completes
the proof.
\end{proof}

Similar to the slope $\mu_C$, we define the slope $\nu_{C}^{\alpha,
\beta}$ of an object $\mathcal{E}\in \Coh_C^{ \beta H}(\mathcal{X})$
by

\begin{eqnarray*}
\nu_{C}^{\alpha, \beta}(\mathcal{E})= \left\{
\begin{array}{ll}
\frac{FH^{n-3}\ch_2^{\beta}(\mathcal{E})-\frac{1}{2}\alpha^2FH^{n-1}\ch^{\beta}_0(\mathcal{E})}{FH^{n-2}\ch^{\beta}_1(\mathcal{E})},
&\mbox{if}~FH^{n-2}\ch^{\beta}_1(\mathcal{E})\neq0,\\
\frac{H^{n-3}\ch_3^{\beta}(\mathcal{E})-\frac{1}{2}\alpha^2H^{n-1}\ch^{\beta}_1(\mathcal{E})}{H^{n-2}\ch^{\beta}_2(\mathcal{E})},
&\mbox{if}~\mathcal{E}_{K(C)}=0~\mbox{and}~H^{n-2}\ch^{\beta}_2(\mathcal{E})\neq0,\\
+\infty, &\mbox{otherwise}.
\end{array}\right.
\end{eqnarray*}
For $\mathcal{E}\in\D^b(\mathcal{X})$, we define

$$Z_{K(C)}^{\alpha, \beta}(\mathcal{E}):=\frac{\alpha^2 }{2}FH^{n-1}\ch_0^{\beta}
(\mathcal{E})-FH^{n-3}\ch_2^{\beta}(\mathcal{E})+i FH^{n-2}\ch_1^{\beta}(\mathcal{E})
$$
and
$$Z_{C\text{-tor}}^{\alpha, \beta}(\mathcal{E}):=\frac{\alpha^2 }{2}H^{n-1}\ch_1^{\beta}
(\mathcal{E})-H^{n-3}\ch_3^{\beta}(\mathcal{E})+i
H^{n-2}\ch_2^{\beta}(\mathcal{E}),$$ respectively. Then by Lemma
\ref{Chern}, one sees that
$$Z_{K(C)}^{\alpha, \beta}(\mathcal{E})=Z_{C\text{-tor}}^{\alpha, \beta}(i_{W*}\mathcal{E}_W),$$ for all closed subscheme $W\subset
C$, where $i_W:\mathcal{X}_W\hookrightarrow \mathcal{X}$ is the
embedding of the fiber over $W$. From \cite[Proposition
25.3]{BLMNPS}, it follows that $(Z_{K(C)}^{\alpha, \beta},
Z_{C\text{-tor}}^{\alpha, \beta}, \Coh^{\beta H}_C(\mathcal{X}))$ is
a weak Harder-Narasimhan structure on $\D^b(\mathcal{X})$ over $C$.
This gives the notion of $\nu_{C}^{\alpha,\beta}$-stability via the
equality
\begin{eqnarray*}
\nu_{C}^{\alpha,\beta}(\mathcal{E})= \left\{
\begin{array}{ll}
-\frac{\Re
Z_{K(C)}^{\alpha,\beta}(\mathcal{E})}{\Im Z_{K(C)}^{\alpha,\beta}(\mathcal{E})}, &\mbox{if}~\mathcal{E}_{K(C)}\neq0,\\
-\frac{\Re Z_{C\text{-tor}}^{\alpha,\beta}(\mathcal{E})}{\Im
Z_{C\text{-tor}}^{\alpha,\beta}(\mathcal{E})},&\mbox{otherwise}.
\end{array}\right.
\end{eqnarray*}


By \cite[Lemma 15.7]{BLMNPS}, one sees that
$\nu_{C}^{\alpha,\beta}$-stability requires stability for all
fibers:
\begin{proposition}
Let $\mathcal{E}\in \Coh_C^{ \beta H}(\mathcal{X})$ be a $C$-torsion
free object. Then $\mathcal{E}$ is
$\nu_{C}^{\alpha,\beta}$-semistable if and only if $\mathcal{E}$ is
$\nu_{H,F}^{\alpha,\beta}$-semistable and for any closed point $p\in
C$ and any quotient $\mathcal{E}_p\twoheadrightarrow \mathcal{Q}$ in
$\Coh^{ \beta H_p}(\mathcal{X}_p)$ we have
$\nu_{p}^{\alpha,\beta}(\mathcal{E}_p)\leq\nu_{p}^{\alpha,\beta}(\mathcal{Q})$.
\end{proposition}

One can translate some basic properties of tilt-stability
(\cite[Proposition 14.2]{Bri2} and \cite[Proposition 7.2.1]{BMT})
into relative tilt-stability.
\begin{lemma}\label{large}
Let $\mathcal{E}\in \Coh_C^{ \beta H}(\mathcal{X})$ be a $C$-torsion
free object.
\begin{enumerate}
\item If $\mathcal{E}$ is $\nu_{H,F}^{\alpha,
\beta}$-semistable  for $\alpha\gg0$, then it satisfies one of the
following conditions:\begin{enumerate}
\item $\mathcal{H}^{-1}(\mathcal{E})=0$ and
$\mathcal{H}^{0}(\mathcal{E})$ is a $\mu_{H,F}$-semistable torsion
free sheaf.
\item $\mathcal{H}^{-1}(\mathcal{E})=0$ and
$\mathcal{H}^{0}(\mathcal{E})$ is a torsion sheaf.
\item $\mathcal{H}^{-1}(\mathcal{E})$ is a $\mu_{H,F}$-semistable torsion
free sheaf and $\mathcal{H}^{0}(\mathcal{E})$ is a torsion sheaf
with $FH^{n-2}\ch_1(\mathcal{H}^{0}(\mathcal{E}))=0$.
\end{enumerate}

\item Let $\mathcal{F}$ be a $\mu_{C}$-stable locally free sheaf on
$\mathcal{X}$ with
$$\overline{\Delta}_{H,F}(\mathcal{F})=(FH^{n-2}\ch_1(\mathcal{F}))^2-2FH^{n-1}\ch_0(\mathcal{F})\cdot(FH^{n-3}\ch_2(\mathcal{F}))=0.$$
Then $\mathcal{F}$ or $\mathcal{F}[1]$ is a $\nu_{H,F}^{\alpha,
\beta}$-stable object in $\Coh_C^{ \beta H}(\mathcal{X})$.
\end{enumerate}
\end{lemma}
\begin{proof}
Assume that $\mathcal{E}$ is $\nu_{H,F}^{\alpha, \beta}$-semistable
for $\alpha\gg0$. One has the following exact sequence in $\Coh_C^{
\beta H}(\mathcal{X})$:
$$0\rightarrow\mathcal{H}^{-1}(\mathcal{E})[1]\rightarrow \mathcal{E}\rightarrow\mathcal{H}^{0}(\mathcal{E})\rightarrow0,$$
where $\mathcal{H}^{-1}(\mathcal{E})\in\mathcal{F}_{\beta H}$ and
$\mathcal{H}^{0}(\mathcal{E})\in\mathcal{T}_{\beta H}$. If
$\mathcal{H}^{-1}(\mathcal{E})=0$ and
$\ch_0(\mathcal{H}^{0}(\mathcal{E}))\neq0$, it is easy to see
$\mathcal{H}^{0}(\mathcal{E})$ is a $\mu_{H,F}$-semistable torsion
free sheaf by the definition $\nu_{H,F}^{\alpha,\beta}$.

Now we assume that $\mathcal{H}^{-1}(\mathcal{E})\neq0$. It turns
out that
$$\nu_{H,F}^{\alpha,\beta}(\mathcal{H}^{-1}(\mathcal{E})[1])\leq\nu_{H,F}^{\alpha,\beta}(\mathcal{E})\leq\nu_{H,F}^{\alpha,\beta}(\mathcal{H}^{0}(\mathcal{E}))$$
for $\alpha\gg0$. This implies
\begin{equation}\label{3.2}
-\frac{FH^{n-1}\ch^{\beta}_0(\mathcal{H}^{-1}(\mathcal{E}))}{FH^{n-2}\ch^{\beta}_1(\mathcal{H}^{-1}(\mathcal{E}))}
\leq-\frac{FH^{n-1}\ch^{\beta}_0(\mathcal{E})}{FH^{n-2}\ch^{\beta}_1(\mathcal{E})}
\leq-\frac{FH^{n-1}\ch^{\beta}_0(\mathcal{H}^{0}(\mathcal{E}))}{FH^{n-2}\ch^{\beta}_1(\mathcal{H}^{0}(\mathcal{E}))}.
\end{equation}
Since $\mathcal{E}$ is $C$-torsion free, so is
$\mathcal{H}^{-1}(\mathcal{E})$. Hence
$\mathcal{H}^{-1}(\mathcal{E})$ is torsion free with
$\mu^+_{H,F}(\mathcal{H}^{-1}(\mathcal{E}))\leq\beta$. From
(\ref{3.2}) and $\mu^-_C(\mathcal{H}^{0}(\mathcal{E}))>\beta$, one
obtains
$$FH^{n-2}\ch^{\beta}_1(\mathcal{H}^{0}(\mathcal{E}))=\ch_0(\mathcal{H}^{0}(\mathcal{E}))=0.$$
Thus one sees that
\begin{equation*}
-\frac{FH^{n-1}\ch^{\beta}_0(\mathcal{H}^{-1}(\mathcal{E}))}{FH^{n-2}\ch^{\beta}_1(\mathcal{H}^{-1}(\mathcal{E}))}
=-\frac{FH^{n-1}\ch^{\beta}_0(\mathcal{E})}{FH^{n-2}\ch^{\beta}_1(\mathcal{E})}.
\end{equation*}
For any subsheaf $\mathcal{K}\subset\mathcal{H}^{-1}(\mathcal{E})$,
we have
$\nu_{H,F}^{\alpha,\beta}(\mathcal{K}[1])\leq\nu_{H,F}^{\alpha,\beta}(\mathcal{E})$
for $\alpha\gg0$. This implies
$$-\frac{FH^{n-1}\ch^{\beta}_0(\mathcal{H}^{-1}(\mathcal{E}))}{FH^{n-2}\ch^{\beta}_1(\mathcal{H}^{-1}(\mathcal{E}))}
=-\frac{FH^{n-1}\ch^{\beta}_0(\mathcal{E})}{FH^{n-2}\ch^{\beta}_1(\mathcal{E})}
\geq-\frac{FH^{n-1}\ch^{\beta}_0(\mathcal{K})}{FH^{n-2}\ch^{\beta}_1(\mathcal{K})}.$$
Hence $\mathcal{H}^{-1}(\mathcal{E})$ is $\mu_{H,F}$-semistable.
This concludes the first statement.


For the second statement, one notices that the $\mu_{C}$-stability
of $\mathcal{F}$ implies that $\mathcal{F}$ or $\mathcal{F}[1]$ is
an object in $\Coh_C^{ \beta H}(\mathcal{X})$, and $\mathcal{F}_s$
is $\mu_{H_s}$-stable for a general $s\in C$ by Lemma
\ref{f-stable}. By \cite[Proposition 7.4.1]{BMT}, one obtains the
$\nu_s^{\alpha, \beta}$-stability of $\mathcal{F}_s$ or
$\mathcal{F}_s[1]$. Hence Lemma \ref{K-stable} gives the second
argument.
\end{proof}

We now give the Bogomolov-Gieseker type inequality for relative
tilt-stable complexes.

\begin{theorem}\label{thm3.8}
If $\mathcal{E}\in\Coh_C^{\beta H}(\mathcal{X})$ is
$\nu_{H,F}^{\alpha,\beta}$-semistable and $n\geq3$, then
$$\overline{\Delta}_{H,F}(\mathcal{E})\geq0.$$
\end{theorem}
\begin{proof}
If $FH^{n-2}\ch^{\beta}_1(\mathcal{E})=0$, by Lemma \ref{ch2} one
easily verifies that $\mathcal{E}$ satisfies the conclusion.

Now we assume that $FH^{n-2}\ch^{\beta}_1(\mathcal{E})>0$. Since the
$\nu_{H,F}^{\alpha,\beta}$-slope of any $C$-torsion object in
$\Coh_C^{\beta H}(\mathcal{X})$ is $+\infty$, by the
$\nu_{H,F}^{\alpha,\beta}$-semistability of $\mathcal{E}$, one
infers that $\mathcal{E}$ is $C$-torsion free. Hence Lemma
\ref{K-stable} gives the $\nu_s^{\alpha,\beta}$-semistability of
$\mathcal{E}_s$ for general $s\in C$. From \cite[Theorem 7.3.1]{BMT}
(see also \cite[Theorem 3.5]{BMS}), it follows that
$$\overline{\Delta}_{H,F}(\mathcal{E})=(H_s^{n-2}\ch_1(\mathcal{E}_s))^2-2H_s^{n-1}\ch_0(\mathcal{E}_s)\cdot(H_s^{n-3}\ch_2(\mathcal{E}_s))\geq0.$$
\end{proof}

\begin{corollary}
If $n\geq3$, then for any $(\alpha, \beta)\in
\sqrt{\mathbb{Q}_{>0}}\times\mathbb{Q}$, the very weak stability
condition $\sigma_{H,F}^{\alpha, \beta}=(Z_{H,F}^{\alpha, \beta},
\Coh_C^{\beta H}(\mathcal{X}))$ satisfies the support property.
\end{corollary}
\begin{proof}
We let $Q$ be the quadratic form
\begin{eqnarray*}
Q:=
(FH^{n-2}\ch^{\beta}_1)^2-2(FH^{n-1}\ch^{\beta}_0)(FH^{n-3}\ch^{\beta}_2)
\end{eqnarray*}
on $\Lambda:=\mathbb{Z}\oplus\mathbb{Z}\oplus\frac{1}{2}\mathbb{Z}$.
Then from Theorem \ref{thm3.8}, one deduces that $Q$ satisfies the
conditions in Definition \ref{supp}. Hence $\sigma_{H,F}^{\alpha,
\beta}$ satisfies the support property.
\end{proof}


\section{Mixed tilt-stability}\label{S4}
In this section, we will introduce the mixed tilt-stability. We keep
the same notations as that in the previous sections.

Let $t$ be a non-negative rational number. Consider the following
central charge
$$z_{\alpha,\beta,t}(\mathcal{E})=\frac{(t+1)\alpha^2 }{2}FH^{n-1}\ch_0(\mathcal{E})-(H^{n-2}+tFH^{n-3})\ch_2^{\beta}(\mathcal{E})
+i FH^{n-2}\ch_1^{\beta}(\mathcal{E}).$$ We think of it as the
composition
$$z_{\alpha,\beta,t}: \K(\D^b(\mathcal{X}))\xrightarrow{v_t} \mathbb{Z}\oplus\mathbb{Z}\oplus\mathbb{Z}\oplus\frac{1}{2}\mathbb{Z}
\xrightarrow{Z_{\alpha,\beta,t}} \mathbb{C},$$ where the first map
is given by
$$v_t(\mathcal{E})=\left(FH^{n-1}\ch_0(\mathcal{E}), FH^{n-2}\ch_1(\mathcal{E}), H^{n-1}\ch_1(\mathcal{E}), (H^{n-2}+tFH^{n-3})\ch_2(\mathcal{E})\right),$$
and the second map is defined by
\begin{eqnarray*}
Z_{\alpha,\beta,t}(e_0, e_1, e^{\prime}_1,
e_2)&=&\frac{1}{2}\left((t+1)\alpha^2-\Big(\frac{H^n}{FH^{n-1}}+t\Big)\beta^2\right)e_0\\
&&+\beta(e^{\prime}_1+te_1)-e_2+i(e_1-\beta e_0).
\end{eqnarray*}

\begin{theorem}\label{tilt2}
For any $(\alpha,\beta, t)\in
\sqrt{\mathbb{Q}_{>0}}\times\mathbb{Q}\times\mathbb{Q}_{\geq0}$,
$\sigma_{\alpha,\beta,t}=(Z_{\alpha,\beta,t}, \Coh_C^{\beta
H}(\mathcal{X}))$ is a very weak stability condition.
\end{theorem}
\begin{proof}
The positivity property follows from Lemma \ref{ch2}. The proof of
the Harder-Narasimhan property is the same as that of Theorem
\ref{tilt}.
\end{proof}

We write $\nu_{\alpha,\beta,t}$ for the slope function on $\Coh_C^{
\beta H}(\mathcal{X})$ induced by $Z_{\alpha,\beta,t}$. Explicitly,
for any $\mathcal{E}\in \Coh_C^{ \beta H}(\mathcal{X})$, one has
\begin{eqnarray*}
\nu_{\alpha,\beta,t}(\mathcal{E})= \left\{
\begin{array}{lcl}
+\infty,  & &\mbox{if}~FH^{n-2}\ch^{\beta}_1(\mathcal{E})=0,\\
&&\\
\frac{(H^{n-2}+tFH^{n-3})\ch_2^{\beta}(\mathcal{E})-\frac{t+1}{2}\alpha^2FH^{n-1}\ch_0(\mathcal{E})}{FH^{n-2}\ch^{\beta}_1(\mathcal{E})},
& &\mbox{otherwise}.
\end{array}\right.
\end{eqnarray*}
Theorem \ref{tilt2} gives the notion of
$\nu_{\alpha,\beta,t}$-stability. Since
$$\nu_{\alpha,\beta,t}=\nu_{\alpha,\beta,0}+t\nu_{H,F}^{\alpha,\beta},$$  we also call
$\nu_{\alpha,\beta,t}$-stability mixed tilt-stability.


\begin{lemma}\label{large2}
Let $\mathcal{E}\in \Coh_C^{ \beta H}(\mathcal{X})$ be a $C$-torsion
free object. If $\mathcal{E}$ is $\nu_{\alpha,\beta,t}$-semistable
for $\alpha\gg0$, then it satisfies one of the following
conditions:\begin{enumerate}
\item $\mathcal{H}^{-1}(\mathcal{E})=0$ and
$\mathcal{H}^{0}(\mathcal{E})$ is a $\mu_{H,F}$-semistable torsion
free sheaf.
\item $\mathcal{H}^{-1}(\mathcal{E})=0$ and
$\mathcal{H}^{0}(\mathcal{E})$ is a torsion sheaf.
\item $\mathcal{H}^{-1}(\mathcal{E})$ is a $\mu_{H,F}$-semistable torsion
free sheaf and $\mathcal{H}^{0}(\mathcal{E})$ is a torsion sheaf
with $FH^{n-2}\ch_1(\mathcal{H}^{0}(\mathcal{E}))=0$.
\end{enumerate}
\end{lemma}
\begin{proof}
The proof is the same as that of Lemma \ref{large}.
\end{proof}

We now show the Bogomolov-Gieseker type inequality of mixed
tilt-stable complexes.

\begin{theorem}\label{Bog2}
Let $\mathcal{E}\in\Coh_C^{\beta H}(\mathcal{X})$ be a
$\nu_{\alpha,\beta, t}$-semistable object, and set $H_t=H+tF$. Then
we have
\begin{eqnarray*}
\widetilde{\Delta}^{\beta H}_{H,F,t}(\mathcal{E})
&:=&(FH^{n-2}\ch_1^{\beta}(\mathcal{E}))(H_tH^{n-2}\ch_1^{\beta}(\mathcal{E}))\\
&&-FH^{n-1}\ch^{\beta}_0(\mathcal{E})\cdot(H_tH^{n-3}\ch^{\beta}_2(\mathcal{E}))\\
&=& \widetilde{\Delta}^{\beta H}_{H,F}(\mathcal{E})+
\frac{t}{2}\overline{\Delta}^{\beta
H}_{H,F}(\mathcal{E})+\frac{t}{2}(FH^{n-2}\ch_1^{\beta}(\mathcal{E}))^2\\
&\geq&0.
\end{eqnarray*}
\end{theorem}
\begin{proof}
The proof is a mimic of that of \cite[Theorem 3.5]{BMS}. We proceed
by induction on $FH^{n-2}\ch^{\beta}_1(\mathcal{E})$, which is a
non-negative function with discrete values on objects of
$\Coh_C^{\beta H}(\mathcal{X})$.

In the case of $FH^{n-2}\ch^{\beta}_1(\mathcal{E})=0$, by Lemma
\ref{ch2} one infers $\widetilde{\Delta}^{\beta
H}_{H,F,t}(\mathcal{E})\geq0$. Now we assume that
$FH^{n-2}\ch^{\beta}_1(\mathcal{E})>0$. Thus $\mathcal{E}$ is
$C$-torsion free. We start increasing $\alpha$. If $\mathcal{E}$
remains stable as $\alpha\rightarrow +\infty$, by Lemma
\ref{large2}, one sees that one of the following holds:
\begin{enumerate}
\item $\mathcal{E}$ is a $\mu_{H,F}$-semistable torsion
free sheaf.
\item $\mathcal{E}$ is a torsion
sheaf.
\item $\mathcal{H}^{-1}(\mathcal{E})$ is a $\mu_{H,F}$-semistable torsion
free sheaf and $\mathcal{H}^{0}(\mathcal{E})$ is a torsion sheaf
with $FH^{n-2}\ch_1(\mathcal{H}^{0}(\mathcal{E}))=0$.
\end{enumerate}
One easily verifies that $\mathcal{E}$ satisfies the conclusion in
any of the possible cases by Theorem \ref{Bog1} and Lemma \ref{ch2}.

Otherwise, $\mathcal{E}$ will get destabilized for some
$\alpha_1>\alpha$ with $\alpha_1^2\in\mathbb{Q}$. Consider the set
$$W:=\{z_{\alpha_1,\beta,t}(\mathcal{K}): 0\neq \mathcal{K}\subset\mathcal{E}~\mbox{and}~
\nu_{\alpha_1,\beta,t}(\mathcal{K})>\nu_{\alpha_1,\beta,t}(\mathcal{E})\}.$$
Since $\nu_{\alpha_1,\beta,t}(\mathcal{K})\leq
\nu^+_{\alpha_1,\beta,t}(\mathcal{E})$,
$FH^{n-2}\ch^{\beta}_1(\mathcal{K})\leq
FH^{n-2}\ch^{\beta}_1(\mathcal{E})$ and the image of
$z_{\alpha_1,\beta,t}$ is discrete, one sees that $W$ is a finite
subset of $\mathbb{C}$. For any element $w\in W$, we set
$$M_w=\{\mathcal{K}: \mathcal{K}\subset
\mathcal{E}~\mbox{and}~z_{\alpha_1,\beta,t}(\mathcal{K})=w\}.$$ By
the discreteness of $z_{\alpha,\beta,t}$, one can find
$\mathcal{K}_w\in M_w$ with
$$\nu_{\alpha,\beta,t}(\mathcal{K}_w)=\max_{\mathcal{K}\in
M_w}\nu_{\alpha,\beta,t}(\mathcal{K}).$$ Since
$\nu_{\alpha,\beta,t}(\mathcal{K}_w)\leq\nu_{\alpha,\beta,t}(\mathcal{E})$
and
$\nu_{\alpha_1,\beta,t}(\mathcal{K}_w)>\nu_{\alpha_1,\beta,t}(\mathcal{E})$,
we can find $\alpha_w\in\sqrt{\mathbb{Q}_{>0}}$ such that
$\alpha\leq\alpha_w<\alpha_1$ and
$\nu_{\alpha_w,\beta,t}(\mathcal{K}_w)=\nu_{\alpha_w,\beta,t}(\mathcal{E})$.
This implies
$\nu_{\alpha_w,\beta,t}(\mathcal{K})\leq\nu_{\alpha_w,\beta,t}(\mathcal{E})$
for any $\mathcal{K}\in M_w$. Taking $\alpha_0=\min_{w\in W}
\alpha_w$, since $\nu_{\alpha,\beta,t}$ is a linear function of
$\alpha$, one can easily check that $\mathcal{E}$ is strictly
$\nu_{\alpha_0,\beta,t}$-semistable. Let
$$0\rightarrow
\mathcal{E}_1\rightarrow\mathcal{E}\rightarrow\mathcal{E}_2\rightarrow0$$
be a short exact sequence where both $\mathcal{E}_1$ and
$\mathcal{E}_2$ have the same $\nu_{\alpha_0,\beta, t}$ slope. Since
both $\mathcal{E}_1$ and $\mathcal{E}_2$ have strictly smaller
$FH^{n-2}\ch^{\beta}_1$, by the induction assumption we have
$\widetilde{\Delta}^{\beta H}_{H, F, t}(\mathcal{E}_1)\geq0$ and
$\widetilde{\Delta}^{\beta H}_{H, F, t}(\mathcal{E}_2)\geq0$.

On the other hand, we think of $\widetilde{\Delta}^{\beta H}_{H,F,
t}$ as a composition
$$\K(\D^b(\mathcal{X}))\xrightarrow{v_t^{\beta}} \mathbb{R}\oplus\NS(\mathcal{X})_{\mathbb{R}}\oplus\mathbb{R} \xrightarrow{q_t^{\beta}}
\mathbb{R},$$ where $v_t^{\beta}$ is given by
$$v_t^{\beta}(\mathcal{E})=(FH^{n-1}\ch_0(\mathcal{E}), \ch^{\beta}_1(\mathcal{E}),
H_tH^{n-3}\ch^{\beta}_2(\mathcal{E}))$$ and $q_t^{\beta}$ is the
quadratic form
$$q_t^{\beta}(r,c,d)=(H_tH^{n-2}c)(FH^{n-2}c)-rd.$$
Let
$Z:\mathbb{R}\oplus\NS(\mathcal{X})_{\mathbb{R}}\oplus\mathbb{R}\rightarrow
\mathbb{C}$ be the linear map defined by
$$Z(r,c,d)=\frac{(t+1)
}{2}\alpha^2r-d +i FH^{n-2}c.$$ It obvious that the kernel of $Z$ is
semi-negative definite with respect to $q_t^{\beta}$. Since
$$v_t^{\beta}(\mathcal{E})=v_t^{\beta}(\mathcal{E}_1)+v_t^{\beta}(\mathcal{E}_2),$$
we deduce $\widetilde{\Delta}^{\beta H}_{H, F, t}(\mathcal{E})\geq0$
by \cite[Lemma 11.7]{BMS}.
\end{proof}

\begin{remark}
I have no examples of quadratic forms satisfies the conditions in
Definition \ref{supp} for $\sigma_{\alpha,\beta, t}$. I do not think
they exist.
\end{remark}


The following lemma gives a relation between relative tilt-stability
and mixed tilt-stability.
\begin{lemma}\label{t-stability}
Assume that $n\geq3$. Let $\mathcal{E}\in \Coh_C^{ \beta
H}(\mathcal{X})$ be a $\nu_{H,F}^{\alpha, \beta}$-stable object.
Then there exists a non-negative rational number $t_0$ only
depending on $\alpha, \beta$ and $\mathcal{E}$, such that
$\mathcal{E}$ is $\nu_{\alpha,\beta,t}$-stable for any $t\geq t_0$.
\end{lemma}
\begin{proof}
Since
$\nu_{\alpha,\beta,t}=\nu_{\alpha,\beta,0}+t\nu_{H,F}^{\alpha,\beta}$,
one sees $\mathcal{E}$ is $\nu_{\alpha,\beta,t}$-stable for any
$t\geq0$ if $\mathcal{E}$ is $\nu_{\alpha,\beta,0}$-stable. Now we
assume that $\mathcal{E}$ is not $\nu_{\alpha,\beta,0}$-stable. Let
$\nu_1$ be the maximal $\nu_{\alpha,\beta,0}$-slope of
$\mathcal{E}$. Then $\nu_1\geq\nu_{\alpha,\beta,0}(\mathcal{E})$.
Let
\begin{eqnarray*}
\nu_2=\max\Big\{\nu_{H,F}^{\alpha,\beta}(\mathcal{F})&:&
\mathcal{F}~\mbox{is a subobject of}~\mathcal{E} ~\mbox{with}\\
&&FH^{n-2}\ch^{\beta}_1(\mathcal{F})<FH^{n-2}\ch^{\beta}_1(\mathcal{E})\Big\}.
\end{eqnarray*}
The $\nu_{H,F}^{\alpha, \beta}$-stability of $\mathcal{E}$ implies
that $\nu_2<\nu_{H,F}^{\alpha,\beta}(\mathcal{E})$. Therefore, when
$$t>\frac{\nu_1-\nu_{\alpha,\beta,0}(\mathcal{E})}{\nu_{H,F}^{\alpha,
\beta}(\mathcal{E})-\nu_2}$$ one sees that
\begin{eqnarray*}
\nu_{\alpha,\beta,t}(\mathcal{E})&=&\nu_{\alpha,\beta,0}(\mathcal{E})+t\nu_{H,F}^{\alpha,\beta}(\mathcal{E})\\
&>&\nu_1+t\nu_2 \\
&\geq&\nu_{\alpha,\beta,0}(\mathcal{F})+t\nu_{H,F}^{\alpha,\beta}(\mathcal{F})\\
&=&\nu_{\alpha,\beta,t}(\mathcal{F})
\end{eqnarray*}
for any subobject $\mathcal{F}\subset\mathcal{E}$ with
$$FH^{n-2}\ch^{\beta}_1(\mathcal{F})<FH^{n-2}\ch^{\beta}_1(\mathcal{E}).$$
This completes the proof.
\end{proof}

Let $\mathbb{D}(-):=\mathbf{R}\mathcal{H}om(-,
\mathcal{O}_{\mathcal{X}})[1]$ denote the duality functor. We will
show that there is a double-dual operation on $\Coh_C^{ \beta
H}(\mathcal{X})$ as well as for coherent sheaves. This is a relative
analogy of \cite[Lemma 2.19]{BLMS}.

\begin{lemma}\label{dual}
Let $\mathcal{E}$ be an object in $\Coh_C^{ \beta H}(\mathcal{X})$
with $\nu^+_{\alpha, \beta,t}(\mathcal{E})<+\infty$ and
$\mathcal{E}^*$ the cohomology object $\mathcal{H}_{\Coh_C^{-\beta
H}(\mathcal{X})}^0(\mathbb{D}(\mathcal{E}))$.
\begin{enumerate}
\item There exists an exact triangle $$\mathcal{E}^*\rightarrow \mathbb{D}(\mathcal{E})\rightarrow
\mathcal{Q}$$ with $\mathcal{H}^j(\mathcal{Q})=0$ for $j\leq0$ and
$\mathcal{H}^j(\mathcal{Q})$ a torsion sheaf supported in
codimension at least $j+2$ for $j\geq1$.
\item There exists an exact sequence in $\Coh_C^{ \beta
H}(\mathcal{X})$
$$0\rightarrow\mathcal{E}\rightarrow\mathcal{E}^{**}\rightarrow\mathcal{E}^{**}/\mathcal{E}\rightarrow0,$$
with $\mathcal{E}^{**}/\mathcal{E}\in\Coh_{\leq n-3}(\mathcal{X})$,
and $\mathcal{E}^{**}$ is quasi-isomorphic to a two term complex
$B^{-1}\rightarrow B^0$ with $B^{-1}$ locally-free and $B^0$
reflexive.
\end{enumerate}
\end{lemma}
\begin{proof}
The proof is similar to that of \cite[Lemma 2.19]{BLMS}. We sketch
it here for reader's convenience. Let
\begin{eqnarray*}
\D_{\pm\beta}^{\leqslant0}&=&\{\mathcal{E}\in\D^b(\mathcal{X}): \mathcal{H}^0(\mathcal{E})\in\mathcal{T}_{\pm\beta H}, \mathcal{H}^i(\mathcal{E})=0~\mbox{for}~i>0\}\\
\D_{\pm\beta}^{\geqslant0}&=&\{\mathcal{E}\in\D^b(\mathcal{X}):
\mathcal{H}^{-1}(\mathcal{E})\in\mathcal{F}_{\pm\beta H},
\mathcal{H}^i(\mathcal{E})=0~\mbox{for}~i<-1\}.
\end{eqnarray*}
By the general theory of torsion pairs and tilting \cite{HRS},
$(\D_{\pm\beta}^{\leqslant0}, \D_{\pm\beta}^{\geqslant0})$ is a
bounded t-structure on $\D^b(\mathcal{X})$. We denote by
$(\tau_{\pm\beta}^{\leqslant0}, \tau_{\pm\beta}^{\geqslant0})$ the
associated truncation functors. We also write $(\D^{\leqslant0},
\D^{\geqslant0})$ for the standard t-structure on
$\D^b(\mathcal{X})$ and $(\tau^{\leqslant0}, \tau^{\geqslant0})$ for
the associated truncation functors.

We first notice that for a coherent sheaf $G\in\Coh(\mathcal{X})$,
the complex $\mathbb{D}(G)$ satisfies
\begin{eqnarray*}
\mathcal{H}^j(\mathbb{D}(G))= \left\{
\begin{array}{lcl}
0,  & &\mbox{if}~j<-1,\\
\mathcal{H}om(G, \mathcal{O}_{\mathcal{X}}), &&\mbox{if}~j=-1,\\
\mathcal{E}xt^{j+1}(G, \mathcal{O}_{\mathcal{X}}),
&&\mbox{if}~j\geq0,
\end{array}\right.
\end{eqnarray*}
where $\mathcal{E}xt^{j+1}(G, \mathcal{O}_{\mathcal{X}})$ is a sheaf
supported in codimension $\geq j+1$. In particular, if $G$ is
supported in codimension $k$, then
$\mathcal{H}^{k-1}(\mathbb{D}(G))$ is the smallest degree with a
nonvanishing cohomology sheaf.

(1) Dualizing the triangle
$\mathcal{H}^{-1}(\mathcal{E})[1]\rightarrow\mathcal{E}\rightarrow\mathcal{H}^0(\mathcal{E})$,
one gets an exact triangle
$$\mathbb{D}(\mathcal{H}^0(\mathcal{E}))\rightarrow\mathbb{D}(\mathcal{E})\rightarrow\mathbb{D}(\mathcal{H}^{-1}(\mathcal{E})[1]).$$
Taking the long exact cohomology sequence, we first see that
$\mathcal{H}^{j}(\mathbb{D}(\mathcal{E}))=0$ for $j<-1$ and
$$\mathcal{H}om(\mathcal{H}^0(\mathcal{E}),
\mathcal{O}_{\mathcal{X}})\cong\mathcal{H}^{-1}(\mathbb{D}(\mathcal{E})).$$
Since $\mu_C^{-}(\mathcal{H}^0(\mathcal{E}))>\beta$, we infer that
$\mathcal{H}^{-1}(\mathbb{D}(\mathcal{E}))$ is either zero or a
torsion free sheaf with
$\mu^+_{C}(\mathcal{H}^{-1}(\mathbb{D}(\mathcal{E})))<-\beta$, i.e.,
$\mathcal{H}^{-1}(\mathbb{D}(\mathcal{E}))\in\mathcal{F}_{-\beta
H}$.

We next obtain a long exact sequence
$$0\rightarrow\mathcal{E}xt^1(\mathcal{H}^0(\mathcal{E}), \mathcal{O}_{\mathcal{X}})\rightarrow\mathcal{H}^0(\mathbb{D}(\mathcal{E}))\rightarrow
\mathcal{H}om(\mathcal{H}^{-1}(\mathcal{E})
,\mathcal{O}_{\mathcal{X}})\xrightarrow{\delta}\mathcal{E}xt^2(\mathcal{H}^0(\mathcal{E}),
\mathcal{O}_{\mathcal{X}}).$$ As
$\nu^+_{\alpha,\beta,t}(\mathcal{E})<+\infty$, one sees that any
subobject of $\mathcal{E}$ has non-zero $H^{n-2}F\ch_1^{\beta}$. So
are the subobjects of $\mathcal{H}^{-1}(\mathcal{E})$. From the
definition of $\mathcal{F}_{\beta H}$ and $\mu_C$, it follows that
$\mathcal{H}^{-1}(\mathcal{E})$ is a torsion free sheaf with
$\mu^+_{C}(\mathcal{H}^{-1}(\mathcal{E}))<\beta$. This implies
$\mu^-_{H,F}(\mathcal{H}om(\mathcal{H}^{-1}(\mathcal{E})
,\mathcal{O}_{\mathcal{X}}))>-\beta$. Since the support of
$\mathcal{E}xt^2(\mathcal{H}^0(\mathcal{E}),
\mathcal{O}_{\mathcal{X}})$ is of codimension at least two, we have
$\mu^-_{H,F}(\ker\delta)>-\beta$. As
$\mathcal{E}xt^1(\mathcal{H}^0(\mathcal{E}),
\mathcal{O}_{\mathcal{X}})$ is a torsion sheaf, one deduces that
$\mathcal{H}^0(\mathbb{D}(\mathcal{E}))_{K(C)}\in\mathcal{T}_{-\beta
H_{K(C)}}$. Let
$T:=\tau_{-\beta}^{\leqslant0}(\mathcal{H}^0(\mathbb{D}(\mathcal{E})))$
be the torsion part of $\mathcal{H}^0(\mathbb{D}(\mathcal{E}))$ with
respect to the torsion pair $(\mathcal{T}_{-\beta H},
\mathcal{F}_{-\beta H})$. Since the Harder-Narasimhan filtration of
$\mathcal{H}^0(\mathbb{D}(\mathcal{E}))$ with respect to the slope
$\mu_C$ induces the Harder-Narasimhan filtration of
$\mathcal{H}^0(\mathbb{D}(\mathcal{E}))_{K(C)}$ with respect to
$\mu_{H_{K(C)}}$, one infers that
$T_{K(C)}=\mathcal{H}^0(\mathbb{D}(\mathcal{E}))_{K(C)}$, and thus
$\mathcal{H}^0(\mathbb{D}(\mathcal{E}))/T$ is a $C$-torsion sheaf in
$\mathcal{F}_{-\beta H}$.

We now consider
$\tau_{-\beta}^{\leqslant0}(\mathbb{D}(\mathcal{E}))$ and
$\mathcal{Q}:=\tau_{-\beta}^{\geqslant1}(\mathbb{D}(\mathcal{E}))$.
By the definition of $\tau_{-\beta}^{\leqslant0}$ and
$\tau_{-\beta}^{\geqslant1}$, one sees
$$T=\mathcal{H}^0(\tau_{-\beta}^{\leqslant0}(\mathbb{D}(\mathcal{E})))\in\mathcal{T}_{-\beta
H}~\mbox{and}~
\mathcal{H}^{0}(\mathcal{Q})=\mathcal{H}^0(\mathbb{D}(\mathcal{E}))/T\in\mathcal{F}_{-\beta
H}.$$ The previous arguments show
$$\mathcal{H}^{-1}(\tau_{-\beta}^{\leqslant0}(\mathbb{D}(\mathcal{E})))=\mathcal{H}^{-1}(\mathbb{D}(\mathcal{E}))\in\mathcal{F}_{-\beta
H},$$ and thus
$$\mathcal{E}^*\cong\tau_{-\beta}^{\leqslant0}(\mathbb{D}(\mathcal{E}))\cong\mathcal{H}_{\Coh_C^{-\beta
H}(\mathcal{X})}^0(\mathbb{D}(\mathcal{E}))\in\Coh_C^{-\beta
H}(\mathcal{X}).$$ It remains to show that
$\mathcal{H}^{j}(\mathbb{D}(\mathcal{E}))=\mathcal{H}^{j}(\mathcal{Q})$
is a torsion sheaf supported in codimension at least $j+2$ for $j>0$
and $\mathcal{H}^{0}(\mathcal{Q})=0$.

The continuation of the long exact cohomology sequence above shows
that $\mathcal{H}^{j}(\mathcal{Q})$ is supported in codimension
$\geq j+1$. Thus we have
$\mathbb{D}(\mathcal{H}^{j}(\mathcal{Q})[-j])\in\D^{\geqslant0}$,
and $\mathcal{H}^{j}(\mathcal{Q})$ is supported in codimension at
least $\geq j+2$ if and only if
$\mathbb{D}(\mathcal{H}^{j}(\mathcal{Q})[-j])\in\D^{\geqslant1}$.
Assume for contradiction that there is a largest possible $j_0>0$
such that
$$\mathcal{H}^0(\mathbb{D}(\mathcal{H}^{j_0}(\mathcal{Q})[-j_0]))\neq0.$$ By induction on the number of non-zero cohomology
objects, we see that $$\mathbb{D}(\tau^{\geqslant k}\mathcal{Q}),~
\mathbb{D}(\tau^{\leqslant k}\mathcal{Q}),
~\mathbb{D}(\mathcal{Q})\in\D^{\geqslant0}$$ for all $k\in
\mathbb{Z}$, $\mathcal{H}^i(\mathbb{D}(\tau^{\geqslant
l}\mathcal{Q}))$ is supported in codimension at least two for any
$i\geq0$ and $l\geq1$, $\mathbb{D}(\tau^{\geqslant
j_0+1}\mathcal{Q})\in\D^{\geqslant1}$ and
$\mathcal{H}^1(\mathbb{D}(\tau^{\geqslant j_0+1}\mathcal{Q}))$ is
supported in codimension at least $j_0+3$. Dualizing the exact
triangle
$$\mathcal{H}^{j_0}(\mathcal{Q})[-j_0]\rightarrow\tau^{\geqslant j_0}\mathcal{Q}\rightarrow\tau^{\geqslant j_0+1}\mathcal{Q}$$
and taking its long exact cohomology sequence, one gets the exact
sequence
\begin{equation*}
0\rightarrow\mathcal{H}^0(\mathbb{D}(\tau^{\geqslant
j_0}\mathcal{Q}))\rightarrow\mathcal{H}^0(\mathbb{D}(\mathcal{H}^{j_0}(\mathcal{Q})[-j_0]))
\rightarrow\mathcal{H}^1(\mathbb{D}(\tau^{\geqslant
j_0+1}\mathcal{Q})).
\end{equation*}
Since the middle object is supported in codimension exactly $j_0+1$,
and the right object is supported in codimension $j_0+3$, it follows
that $\mathcal{H}^0(\mathbb{D}(\tau^{\geqslant
j_0}\mathcal{Q}))\neq0$. Dualizing the exact triangle
$$\tau^{\leqslant
j_0-1}\mathcal{Q}\rightarrow\mathcal{Q}\rightarrow\tau^{\geqslant
j_0}\mathcal{Q}$$ gives an injection
$$0\neq\mathcal{H}^0(\mathbb{D}(\tau^{\geqslant j_0}\mathcal{Q}))\hookrightarrow\mathcal{H}^0(\mathbb{D}(\mathcal{Q})).$$
Similarly, dualizing the exact triangle
$$\mathcal{H}^0(\mathcal{Q})\rightarrow\mathcal{Q}\rightarrow\tau^{\geqslant
1}\mathcal{Q}$$ gives
\begin{equation}\label{3.3}
0\rightarrow\mathcal{H}^0(\mathbb{D}(\tau^{\geqslant1}\mathcal{Q}))\rightarrow\mathcal{H}^0(\mathbb{D}(\mathcal{Q}))
\rightarrow\mathcal{E}xt^1(\mathcal{H}^0(\mathcal{Q}),
\mathcal{O}_{\mathcal{X}})\rightarrow\mathcal{H}^1(\mathbb{D}(\tau^{\geqslant1}\mathcal{Q})),
\end{equation}
thus $\mathcal{H}^0(\mathbb{D}(\mathcal{Q}))$ is a torsion sheaf.
Now consider the exact triangle
$$\mathbb{D}(\mathcal{Q})\rightarrow\mathbb{D}(\mathbb{D}(\mathcal{E}))=\mathcal{E}\rightarrow\mathbb{D}(\mathcal{E}^*).$$
The same arguments as before show
$\mathcal{H}^{-1}(\mathbb{D}(\mathcal{E}^*))$ is a torsion free
sheaf in $\mathcal{F}_{\beta H}$, and hence
$$\Hom(\mathcal{H}^0(\mathbb{D}(\mathcal{Q})), \mathbb{D}(\mathcal{E}^*)[-1])
=\Hom(\mathcal{H}^0(\mathbb{D}(\mathcal{Q})),
\mathcal{H}^{-1}(\mathbb{D}(\mathcal{E}^*)))=0.$$ Therefore, the
composition
$\mathcal{H}^0(\mathbb{D}(\mathcal{Q}))\rightarrow\mathbb{D}(\mathcal{Q})\rightarrow\mathcal{E}$
is non-zero. So is $$\mathcal{H}^0(\mathbb{D}(\tau^{\geqslant
j_0}\mathcal{Q}))\hookrightarrow\mathcal{H}^0(\mathbb{D}(\mathcal{Q}))\rightarrow\mathcal{E}.$$
This is a contradiction to
$\nu^+_{\alpha,\beta,t}(\mathcal{E})<+\infty$.

Now we show that $\mathcal{H}^0(\mathcal{Q})=0$. Assume for
contradiction that $\mathcal{H}^0(\mathcal{Q})\neq0$. Then
$\mathcal{E}xt^1(\mathcal{H}^0(\mathcal{Q}),
\mathcal{O}_{\mathcal{X}})$ is a $C$-torsion sheaf supported in
codimension one. Since
$\mathcal{H}^0(\mathbb{D}(\tau^{\geqslant1}\mathcal{Q}))=0$, from
(\ref{3.3}), it follows that
$\mathcal{H}^0(\mathbb{D}(\mathcal{Q}))$ is $C$-torsion. This is
still a contradiction to
$\nu^+_{\alpha,\beta,t}(\mathcal{E})<+\infty$.

From the above proof, we see that
$$\mathcal{E}^*=\tau^{\leqslant0}(\mathbb{D}(\mathcal{E}))
~\mbox{and}~
\mathcal{Q}=\tau^{\geqslant1}(\mathbb{D}(\mathcal{E})).$$

(2) By the proof of part (1), we have the following diagram of exact
triangles
\begin{equation*}
\xymatrix{& & \mathcal{E}^{**} \ar[d] \\
\mathbb{D}(\mathcal{Q})  \ar[r] &\mathcal{E}   \ar[r]& \mathbb{D}(\mathcal{E}^*) \ar[d] \\
& &\mathcal{Q}^{\prime}
  }
\end{equation*}
with $\mathbb{D}(\mathcal{Q})\in\D^{\geqslant1}$ whose cohomology
sheaves are supported in codimension at least 3,
$\mathcal{Q}^{\prime}\in\D^{\geqslant0}$, and
$\mathcal{H}^0(\mathcal{Q}^{\prime})$ is a $C$-torsion sheaf in
$\mathcal{F}_{\beta H}$, $\mathcal{H}^j(\mathcal{Q}^{\prime})$ is
supported in codimension $\geq j+1$ for $j>0$, whereas $\mathcal{E},
\mathcal{E}^{**}\in\Coh_C^{\beta H}(\mathcal{X})$. Since
$\mathcal{E}\in\D_{\beta}^{\leqslant0}$ and
$\mathcal{Q}^{\prime}=\tau_{\beta}^{\geqslant1}(\mathbb{D}(\mathcal{E}^*))\in\D_{\beta}^{\geqslant1}$,
we have $\Hom(\mathcal{E}, \mathcal{Q}^{\prime})=0$. So we have an
induced morphism $\mathcal{E}\rightarrow\mathcal{E}^{**}$. The cone
$\mathcal{C}$ of this morphism fits into an exact triangle
\begin{equation}\label{3.4}
\mathcal{Q}^{\prime}[-1]\rightarrow\mathcal{C}\rightarrow\mathbb{D}(\mathcal{Q})[1].
\end{equation}
Taking its long exact cohomology sequence, one deduces
that $\mathcal{C}\in\D^{\geqslant0}$ and
$\mathcal{H}^{0}(\mathcal{C})$ is supported in codimension at least
3. The long exact cohomology sequence of the exact triangle
$\mathcal{E}\rightarrow\mathcal{E}^{**}\rightarrow\mathcal{C}$ shows
that $\mathcal{H}^{j}(\mathcal{C})=0$ for $j>0$. Therefore, the
morphism $\mathcal{E}\rightarrow\mathcal{E}^{**}$ is injective in
$\Coh_C^{\beta H}(\mathcal{X})$, and its cokernel is a torsion sheaf
supported in codimension at least 3.

Considering again the long exact cohomology sequence induced by
(\ref{3.4}), one sees that $\mathcal{H}^0(\mathcal{Q}^{\prime})=0$
and
$\mathcal{H}^j(\mathcal{Q}^{\prime})=\mathcal{H}^{j+1}(\mathbb{D}(\mathcal{Q}))$
is supported in codimension at least $3$. Thus, by the proof of part
(1), we have
$$\mathcal{E}^{**}=\tau^{\leqslant0}\mathbb{D}(\mathcal{E}^*)
~\mbox{and}~\mathcal{Q}^{\prime}=\tau^{\geqslant1}\mathbb{D}(\mathcal{E}^*).
$$
To finish the proof, we consider a locally-free resolution
$G^{\bullet}$ of $\mathcal{E}^*$. Taking the functor $\mathbb{D}$,
one obtain a morphism
$$\mathcal{E}^{**}\rightarrow\left(G_0^{\vee}\xrightarrow{\phi_0} G_1^{\vee}\xrightarrow{\phi_1}\cdots\rightarrow G_m^{\vee}\right)[1].$$
Hence $\mathcal{E}^{**}=\tau^{\leqslant0}\mathbb{D}(\mathcal{E}^*)$
is quasi-isomorphic to the complex
$G_0^{\vee}\xrightarrow{\phi_0}\ker \phi_1$, and $\ker \phi_1$ is
reflexive as it is the kernel of a morphism of locally-free sheaves.
\end{proof}

From this lemma, one obtains a relative version of \cite[Proposition
2.18]{BLMS}.
\begin{proposition}\label{dual2}
Let $\mathcal{C}^0\subset \Coh_C^{ \beta H}(\mathcal{X})$ be the
subcategory of objects $\mathcal{E}\in\Coh_C^{ \beta
H}(\mathcal{X})$ with $z_{\alpha,\beta,t}(\mathcal{E})=0$. We have
the following:
\begin{enumerate}
\item For $\mathcal{E}\in\Coh_C^{ \beta H}(\mathcal{X})$, there exists a maximal
subobject $\widetilde{\mathcal{E}}\in\mathcal{C}^0$ of $\mathcal{E}$
such that $\Hom(\mathcal{C}^0,
\mathcal{E}/\widetilde{\mathcal{E}})=0$.
\item For $\mathcal{E}\in\Coh_C^{ \beta H}(\mathcal{X})$ with
$\nu_{\alpha,\beta,t}(\mathcal{E})<+\infty$, there exists a short
exact sequence
$\mathcal{E}\hookrightarrow\mathcal{E}^{**}\twoheadrightarrow\mathcal{E}^0$
with $\mathcal{E}^0\in\Coh_{\leq n-3}(\mathcal{X})$ and
$\Hom(\Coh_{\leq n-3}(\mathcal{X}), \mathcal{E}^{**}[1])=0$.
\end{enumerate}
\end{proposition}
\begin{proof}
It turns out that $\mathcal{C}^0$ is abelian. As we showed in the
proof of Theorem \ref{tilt}, $\Coh_C^{ \beta H}(\mathcal{X})$ is
noetherian. So is $\mathcal{C}^0$. It follows that we can find a
maximal subobject $\widetilde{\mathcal{E}}\in\mathcal{C}^0$ of
$\mathcal{E}$ satisfying property (1).

The proof of property (2) is the same as that of \cite[Proposition
2.18]{BLMS}.
\end{proof}

\section{Construction and conjecture}\label{S5}
In this section, we give the construction of the heart
$\mathcal{A}_{t}^{\alpha,\beta}(\mathcal{X})$ of a bounded
$t$-structure on $\D^b(\mathcal{X})$ as a tilt starting from
$\Coh_C^{\beta H}(\mathcal{X})$, and state our main conjectures. We
always assume that $n=\dim\mathcal{X}=3$ throughout this section.

We consider the torsion pair
$(\mathcal{T}_t^{\prime},\mathcal{F}_t^{\prime})$ in $\Coh_C^{\beta
H}(\mathcal{X})$ as follows:
\begin{eqnarray*}
\mathcal{T}_t^{\prime}&=&\{\mathcal{E}\in\Coh_C^{\beta
H}(\mathcal{X}): \text{any quotient}~
\mathcal{E}\twoheadrightarrow\mathcal{G}~ \text{satisfies}~
\nu_{\alpha,\beta,t}(\mathcal{G})>0 \}\\
\mathcal{F}_t^{\prime}&=&\{\mathcal{E}\in\Coh_C^{\beta
H}(\mathcal{X}):\text{any subobject}~
\mathcal{K}\hookrightarrow\mathcal{E}~ \text{satisfies}~
\nu_{\alpha,\beta,t}(\mathcal{K})\leq0 \}.
\end{eqnarray*}

\begin{definition}
We define the abelian category
$\mathcal{A}_{t}^{\alpha,\beta}(\mathcal{X})\subset
\D^b(\mathcal{X})$ to be the extension-closure
$$\mathcal{A}_{t}^{\alpha,\beta}(\mathcal{X})=\langle\mathcal{T}_t^{\prime},
\mathcal{F}_t^{\prime}[1]\rangle.$$
\end{definition}


We propose the following conjecture which can be considered as a
relative analogy of \cite[Conjecture 1.3.1]{BMT}.

\begin{conjecture}\label{Conj1}
There exists a triple $(\alpha,\beta, t)$ in
$\sqrt{\mathbb{Q}_{>0}}\times\mathbb{Q}\times\mathbb{Q}_{\geq0}$
such that
\begin{equation}\label{BG1}
\ch_3^{\beta }(\mathcal{E})\leq (a_1H^2+b_1HF)\ch^{\beta
}_{1}(\mathcal{E})+(a_2H+b_2F)\ch^{\beta}_2(\mathcal{E})+c\ch^{\beta}_0(\mathcal{E})
\end{equation}
for any $\nu_{\alpha,\beta,t}$-semistable object
$\mathcal{E}\in\Coh_C^{\beta H}(\mathcal{X})$ with
$\nu_{\alpha,\beta,t}(\mathcal{E})=0$, where the constants $a_1$,
$b_1$, $a_2$, $b_2$ and $c$ are independent of $\mathcal{E}$ and
$a_1>0$.
\end{conjecture}




Consider the following central charge
\begin{eqnarray*}
z_{l}&=& (a_1H^2+lHF)\ch^{\beta
}_{1}+(a_2H+b_2F)\ch_2^{\beta}+c\ch^{\beta}_0-\ch_3^{\beta
}\\
&&+i\Big((H+tF)\ch_2^{\beta}(\mathcal{E})-\frac{t+1}{2}\alpha^2FH^{2}\ch_0(\mathcal{E})\Big).
\end{eqnarray*}
We think of it as the composition
$$z_{l}: \K(\D^b(\mathcal{X}))\xrightarrow{\bar{v}} \mathbb{Z}\oplus\mathbb{Z}\oplus\mathbb{Z}
\oplus\frac{1}{2}\mathbb{Z}\oplus\frac{1}{2}\mathbb{Z}\oplus\frac{1}{6}\mathbb{Z}
\xrightarrow{Z_{l}} \mathbb{C},$$ where the first map is given by
$$\bar{v}(\mathcal{E})=\left(\ch_0(\mathcal{E}), FH\ch_1(\mathcal{E}), H^2\ch_1(\mathcal{E}), F\ch_2(\mathcal{E}), H\ch_2(\mathcal{E}),\ch_3(\mathcal{E})\right).$$
Conjecture \ref{Conj1} implies the existence of stability conditions
on $\mathcal{X}$.

\begin{theorem}\label{main}
Assume Conjecture \ref{Conj1} holds for
$(\alpha,\beta,t)\in\sqrt{\mathbb{Q}_{>0}}\times\mathbb{Q}\times\mathbb{Q}_{\geq0}$,
then the pair $(Z_{l},\mathcal{A}_{t}^{\alpha, \beta}(\mathcal{X}))$
is a locally finite stability condition on $\mathcal{X}$ if
$l>\max\{b_1,0\}$.
\end{theorem}
\begin{proof}
\textbf{Step 1.} The pair $(Z_{l},\mathcal{A}_{t}^{\alpha,
\beta}(\mathcal{X}))$ satisfies the positivity property for any
$0\neq\mathcal{E}\in \mathcal{A}_{t}^{\alpha, \beta}(\mathcal{X})$.
\bigskip

The construction of the heart $\mathcal{A}_{t}^{\alpha,
\beta}(\mathcal{X})$ directly ensures that $\Im
Z_{s,t}(\bar{v}(\mathcal{E}))\geq0$ for any
$\mathcal{E}\in\mathcal{A}_{t}^{\alpha, \beta}(\mathcal{X})$.
Moreover, if  $\Im Z_{s,t}(\bar{v}(\mathcal{E}))=0$, then
$\mathcal{E}$ fits into an exact triangle
$$\mathcal{K}[1]\rightarrow\mathcal{E}\rightarrow \mathcal{G}$$
where
\begin{enumerate}
\item $\mathcal{G}$ is an object in $\Coh_C^{\beta
H}(\mathcal{X})$ with
$$HF\ch_1^{\beta}(\mathcal{G})=(H+tF)\ch_2^{\beta}(\mathcal{G})-\frac{t+1}{2}\alpha^2FH^{2}\ch_0(\mathcal{G})=0;$$
\item $\mathcal{K}\in \Coh_C^{\beta
H}(\mathcal{X})$ is $\nu_{\alpha,\beta,t}$-semistable with
$\nu_{\alpha,\beta, t}(\mathcal{K})=0$ and
$HF\ch_1^{\beta}(\mathcal{K})>0$.
\end{enumerate}
By Lemma \ref{ch2}, one sees that
$$HF\ch_1^{\beta}(\mathcal{G})=H\ch_2^{\beta}(\mathcal{G})=F\ch_2^{\beta}(\mathcal{G})=\ch_0(\mathcal{G})=0,$$
$\ch^{\beta}_3(\mathcal{G})\geq0$, $\mathcal{H}^{-1}(\mathcal{G})$
is $C$-torsion and
$\mathcal{H}^0(\mathcal{G})\in\Coh_{\leq0}(\mathcal{X})$. These
imply that $H^2\ch_1(\mathcal{H}^{-1}(\mathcal{G}))>0$ and $\Re
z_l(\mathcal{G})<0$. On the other hand, Conjecture \ref{Conj1}
implies that
\begin{eqnarray*}
\ch_3^{\beta }(\mathcal{K})&\leq& (a_1H^2+b_1HF)\ch^{\beta
}_{1}(\mathcal{K})+(a_2H+b_2F)\ch^{\beta}_2(\mathcal{K})+c\ch^{\beta}_0(\mathcal{K})\\
&<& (a_1H^2+lHF)\ch^{\beta
}_{1}(\mathcal{K})+(a_2H+b_2F)\ch^{\beta}_2(\mathcal{K})+c\ch^{\beta}_0(\mathcal{K}),
\end{eqnarray*}
i.e., $\Re z_l(\mathcal{K})>0$. Therefore one concludes that
$$\Re z_l(\mathcal{E})=\Re z_l(\mathcal{G})-\Re z_l(\mathcal{K})<0.$$

\bigskip
\textbf{Step 2.} The category $\mathcal{A}_{t}^{\alpha,
\beta}(\mathcal{X})$ is noetherian.
\bigskip

We take a chain of surjections in $\mathcal{A}_{t}^{\alpha,
\beta}(\mathcal{X})$:
\begin{equation}\label{5.1}
\mathcal{E}_0\twoheadrightarrow
\mathcal{E}_1\twoheadrightarrow\cdots\twoheadrightarrow
\mathcal{E}_m\twoheadrightarrow \cdots.
\end{equation}
Since $\Im z_l(\mathcal{E}_j)\geq0$ and the image of $z_l$ is
discrete, we can assume that $\Im z_l(\mathcal{E}_j)$ is constant
for all $j\geq0$. Then one obtains short exact sequences in
$\mathcal{A}_{t}^{\alpha, \beta}(\mathcal{X})$:
$$0\rightarrow\mathcal{K}_j\rightarrow\mathcal{E}_0\rightarrow\mathcal{E}_j\rightarrow0$$
with $\Im z_l(\mathcal{K}_j)=0$. By the Noetherianity of
$\Coh_C^{\beta H}(\mathcal{X})$, we may assume that
$$\mathcal{H}^0_{\Coh_C^{\beta H}(\mathcal{X})}(\mathcal{E}_0)=\mathcal{H}^0_{\Coh_C^{\beta
H}(\mathcal{X})}(\mathcal{E}_j)~\mbox{and}~\mathcal{H}^{-1}_{\Coh_C^{\beta
H}(\mathcal{X})}(\mathcal{K}_j)=\mathcal{H}^{-1}_{\Coh_C^{\beta
H}(\mathcal{X})}(\mathcal{K}_{j+1})$$ for any $j\geq1$. Setting
$\mathcal{U}=\mathcal{H}^{-1}_{\Coh_C^{\beta
H}(\mathcal{X})}(\mathcal{E}_0)/\mathcal{H}^{-1}_{\Coh_C^{\beta
H}(\mathcal{X})}(\mathcal{K}_j)$, one gets the short exact sequences
\begin{equation}\label{5.2}
0\rightarrow \mathcal{U}\rightarrow\mathcal{H}^{-1}_{\Coh_C^{\beta
H}(\mathcal{X})}(\mathcal{E}_j)\xrightarrow{g_j}\mathcal{H}^{0}_{\Coh_C^{\beta
H}(\mathcal{X})}(\mathcal{K}_j)\rightarrow0.
\end{equation}
Consider the exact sequence in $\mathcal{A}_{t}^{\alpha,
\beta}(\mathcal{X})$:
\begin{equation}\label{5.3}
0\rightarrow\mathcal{K}_j\rightarrow\mathcal{K}_{j+1}\rightarrow\mathcal{Q}_{j+1}\rightarrow0.
\end{equation}
Since $\Im z_l(\mathcal{K}_j)=0$, by the definition of
$\mathcal{A}_{t}^{\alpha, \beta}(\mathcal{X})$, one sees that
$\mathcal{H}^{0}_{\Coh_C^{\beta H}(\mathcal{X})}(\mathcal{K}_j)$ has
zero $HF\ch_1^{\beta}$, $H\ch_2^{\beta}$, $F\ch_2^{\beta}$ and
$\ch_0$. Lemma \ref{ch2} gives that
$\mathcal{H}^{-1}\left(\mathcal{H}^{0}_{\Coh_C^{\beta
H}(\mathcal{X})}(\mathcal{K}_j)\right)$ is a $C$-torsion
$\mu_C$-semistable sheaf with $\mu_C$-slope $\beta$, and
$\mathcal{H}^{0}\left(\mathcal{H}^{0}_{\Coh_C^{\beta
H}(\mathcal{X})}(\mathcal{K}_j)\right)$ is a sheaf supported in
dimension zero. The same argument holds for
$\mathcal{H}^{0}_{\Coh_C^{\beta
H}(\mathcal{X})}(\mathcal{Q}_{j+1})$. Taking the long exact
cohomology sequence of (\ref{5.3}), one sees that
$\mathcal{H}^{-1}_{\Coh_C^{\beta
H}(\mathcal{X})}(\mathcal{Q}_{j+1})$ is a subobject of
$\mathcal{H}^{0}_{\Coh_C^{\beta H}(\mathcal{X})}(\mathcal{K}_j)$,
and thus $$\mathcal{H}^{-1}_{\Coh_C^{\beta
H}(\mathcal{X})}(\mathcal{Q}_{j+1})=0$$ by the definition of
$\mathcal{A}_{t}^{\alpha, \beta}(\mathcal{X})$. Hence we have a
chain of injections in $\mathcal{T}_t^{\prime}$
$$\mathcal{H}^{0}_{\Coh_C^{\beta H}(\mathcal{X})}(\mathcal{K}_1)\subset\mathcal{H}^{0}_{\Coh_C^{\beta
H}(\mathcal{X})}(\mathcal{K}_2)\subset\cdots.$$ This gives a chain
of injections in $\mathcal{F}_t^{\prime}$:
$$\mathcal{H}^{-1}_{\Coh_C^{\beta H}(\mathcal{X})}(\mathcal{E}_1)\subset\mathcal{H}^{-1}_{\Coh_C^{\beta
H}(\mathcal{X})}(\mathcal{E}_2)\subset\cdots$$ with
$$\mathcal{H}^{-1}_{\Coh_C^{\beta
H}(\mathcal{X})}(\mathcal{E}_{j+1})/\mathcal{H}^{-1}_{\Coh_C^{\beta
H}(\mathcal{X})}(\mathcal{E}_j)\cong\mathcal{H}^{0}_{\Coh_C^{\beta
H}(\mathcal{X})}(\mathcal{Q}_{j+1})\cong\mathcal{Q}_{j+1}$$ for
$j\geq1$.

As we have showed that $\mathcal{Q}_{j+1}\in\Coh_C^{\beta
H}(\mathcal{X})$, $\mathcal{H}^{0}(\mathcal{Q}_{j+1})$ is supported
in dimension zero and $\mathcal{H}^{-1}(\mathcal{Q}_{j+1})$ is a
$C$-torsion $\mu_C$-semistable sheaf with
$\mu_C(\mathcal{H}^{-1}(\mathcal{Q}_{j+1}))=\beta$, one sees that
$H^2\ch_1^{\beta}(\mathcal{Q}_{j+1})<0$ if
$\mathcal{H}^{-1}(\mathcal{Q}_{j+1})\neq0$. Thus we have
$$H^2\ch_1^{\beta}\left(\mathcal{H}^{-1}_{\Coh_C^{\beta
H}(\mathcal{X})}(\mathcal{E}_{j})\right)>H^2\ch_1^{\beta}\left(\mathcal{H}^{-1}_{\Coh_C^{\beta
H}(\mathcal{X})}(\mathcal{E}_{j+1})\right)$$ if
$\mathcal{H}^{-1}(\mathcal{Q}_{j+1})\neq0$. As the proof of
\cite[Lemma 2.15]{PT}, by induction on the number of
Harder-Narasimhan factors of $\mathcal{H}^{-1}_{\Coh_C^{\beta
H}(\mathcal{X})}(\mathcal{E}_{1})$ with respect to
$\nu_{\alpha,\beta,t}$, one may assume that
$\mathcal{H}^{-1}_{\Coh_C^{\beta H}(\mathcal{X})}(\mathcal{E}_{1})$
is $\nu_{\alpha,\beta,t}$-semistable. Hence one infers
$\mathcal{H}^{-1}_{\Coh_C^{\beta H}(\mathcal{X})}(\mathcal{E}_{j})$
is $\nu_{\alpha,\beta,t}$-semistable with positive $HF\ch_1^{\beta}$
for any $j\geq1$ by \cite[Sublemma 2.16]{PT}. Since
$HF\ch_1^{\beta}$, $H\ch_2^{\beta}$, $F\ch_2^{\beta}$ and $\ch_0$ of
$\mathcal{H}^{-1}_{\Coh_C^{\beta H}(\mathcal{X})}(\mathcal{E}_{j})$
are constant as $j$ grows, by Theorem \ref{Bog2} one deduces that
there is a rational number $c_0$ such that
$$H^2\ch_1^{\beta}\left(\mathcal{H}^{-1}_{\Coh_C^{\beta
H}(\mathcal{X})}(\mathcal{E}_{j})\right)\geq c_0$$ for any $j\geq1$.
This implies that
$$H^2\ch_1^{\beta}\left(\mathcal{H}^{-1}_{\Coh_C^{\beta
H}(\mathcal{X})}(\mathcal{E}_{j})\right)=H^2\ch_1^{\beta}\left(\mathcal{H}^{-1}_{\Coh_C^{\beta
H}(\mathcal{X})}(\mathcal{E}_{j+1})\right)$$ when $j\gg0$.
Therefore, we conclude that $\mathcal{H}^{-1}(\mathcal{Q}_{j+1})=0$
when $j\gg0$. We may assume that $\mathcal{Q}_{j+1}$ is a sheaf
supported in dimension zero for any $j\geq1$.

Taking the long exact cohomology sequence of (\ref{5.3}) again, one
obtains exact sequences of sheaves supported in dimension zero
$$0\rightarrow\mathcal{H}^0\left(\mathcal{H}^{0}_{\Coh_C^{\beta H}(\mathcal{X})}(\mathcal{K}_j)\right)\rightarrow\mathcal{H}^0\left(\mathcal{H}^{0}_{\Coh_C^{\beta
H}(\mathcal{X})}(\mathcal{K}_{j+1})\right)\rightarrow\mathcal{Q}_{j+1}\rightarrow0$$
and equalities
$$\mathcal{H}^{-1}\left(\mathcal{H}^{0}_{\Coh_C^{\beta H}(\mathcal{X})}(\mathcal{K}_1)\right)=\mathcal{H}^{-1}\left(\mathcal{H}^{0}_{\Coh_C^{\beta
H}(\mathcal{X})}(\mathcal{K}_{2})\right)=\cdots.$$ Let
$\mathcal{G}_i$ be the kernel of the composition
$$\mathcal{H}^{-1}_{\Coh_C^{\beta
H}(\mathcal{X})}(\mathcal{E}_j)\xrightarrow{g_j}\mathcal{H}^{0}_{\Coh_C^{\beta
H}(\mathcal{X})}(\mathcal{K}_j)\rightarrow\mathcal{H}^0\left(\mathcal{H}^{0}_{\Coh_C^{\beta
H}(\mathcal{X})}(\mathcal{K}_j)\right).$$ Then we have the following
commutative diagram
\begin{equation*}
\xymatrix{
0\ar[r]&\mathcal{G}_j\ar[r]\ar@{==}[d]&\mathcal{H}^{-1}_{\Coh_C^{\beta
H}(\mathcal{X})}(\mathcal{E}_j) \ar@{^(->}[d]\ar[r]&
\mathcal{H}^0\left(\mathcal{H}^{0}_{\Coh_C^{\beta
H}(\mathcal{X})}(\mathcal{K}_j)\right) \ar@{^(->}[d]
\ar[r] & 0 \\
0\ar[r]&\mathcal{G}_{j+1}\ar[r]& \mathcal{H}^{-1}_{\Coh_C^{\beta
H}(\mathcal{X})}(\mathcal{E}_{j+1})\ar[r]\ar@{->>}[d]&
\mathcal{H}^0\left(\mathcal{H}^{0}_{\Coh_C^{\beta
H}(\mathcal{X})}(\mathcal{K}_j)\right) \ar[r]\ar@{->>}[d]& 0 \\
&& \mathcal{Q}_{j+1} \ar@{=}[r]& \mathcal{Q}_{j+1} &}
\end{equation*}
The snake lemma gives $\mathcal{G}_1=\mathcal{G}_{j}$ for any
$j\geq1$. By Proposition \ref{dual2}, one concludes that
$\mathcal{H}^0\left(\mathcal{H}^{0}_{\Coh_C^{\beta
H}(\mathcal{X})}(\mathcal{K}_j)\right)$ is a subsheaf of the zero
dimensional sheaf $\mathcal{G}_1^{**}/\mathcal{G}_1$. In particular
the degree of $\mathcal{H}^0\left(\mathcal{H}^{0}_{\Coh_C^{\beta
H}(\mathcal{X})}(\mathcal{K}_j)\right)$ is bounded. This shows that
$\mathcal{Q}_j=0$ for large $j$, and hence
$\mathcal{K}_j=\mathcal{K}_{j+1}$ for large $j$. This implies that
the chain (\ref{5.1}) terminates, and thus $\mathcal{A}_{t}^{\alpha,
\beta}(\mathcal{X})$ is Noetherian.

\bigskip
\textbf{Step 3.} The pair $(Z_{l},\mathcal{A}_{t}^{\alpha,
\beta}(\mathcal{X}))$ satisfies the
Harder-Narasimhan property and the local finiteness property.
\bigskip

Since $\mathcal{A}_{t}^{\alpha,
\beta}(\mathcal{X})$ is Noetherian,
from the discreteness of $Z_l$,
we conclude that $(Z_{l},\mathcal{A}_{t}^{\alpha,
\beta}(\mathcal{X}))$ satisfies the Harder-Narasimhan property.
The local finiteness follows immediately from \cite[Lemma
4.4]{Bri2}.
\end{proof}

\begin{remark}
I do not know whether $(Z_{l},\mathcal{A}_{t}^{\alpha,
\beta}(\mathcal{X}))$ satisfies the support property.
\end{remark}

\section{Stability conditions on projective bundles}\label{S6}
Throughout this section we let $E$ be a locally free sheaf on $C$
with $\rank E=3$ and $\mathcal{X}:=\mathbb{P}(E)$ be the projective
bundle associated to $E$ with the projection
$f:\mathcal{X}\rightarrow C$ and the associated relative ample
invertible sheaf $\mathcal{O}_{\mathcal{X}}(1)$. Since
$\mathbb{P}(E)\cong \mathbb{P}(E\otimes L)$ for any line bundle $L$
on $C$, we can assume that $H:=c_1(\mathcal{O}_{\mathcal{X}}(1))$ is
ample. One sees that $H^3=\deg E$ and $H^2F=1$. We freely use the
notations in previous sections.

\begin{lemma}\label{can}
Denote by $g$ the genus of $C$. Then we have
\begin{eqnarray*}
c_1(T_{\mathcal{X}})&=&-f^*K_C-f^*c_1(E)+3H\\
c_2(T_{\mathcal{X}})&=& 3H^2-(6g-6+2\deg E)HF.
\end{eqnarray*}
In particular, for any divisor $D$ on $\mathcal{X}$ we have
\begin{eqnarray*}
DFc_1(T_{\mathcal{X}})&=&3DHF\\
DHc_1(T_{\mathcal{X}})&=&3DH^2-(2g-2+H^3)DHF\\
D\Big(c^2_1(T_{\mathcal{X}})+c_2(T_{\mathcal{X}})\Big)&=&
12DH^2-(18g-18+8H^3)DHF\\
\chi(\mathcal{O}_{\mathcal{X}})&=&\frac{1}{24}c_1(T_{\mathcal{X}})c_2(T_{\mathcal{X}})=1-g.
\end{eqnarray*}
\end{lemma}
\begin{proof}
The formulas follow from the relative Euler sequence
$$0\rightarrow \Omega_{\mathcal{X}/C}\rightarrow f^*\mathcal{E}\otimes\mathcal{O}_{\mathcal{X}}(-H)\rightarrow\mathcal{O}_{\mathcal{X}}\rightarrow0$$
and the standard exact sequence
$$0\rightarrow f^*\omega_C\rightarrow\Omega_{\mathcal{X}}\rightarrow\Omega_{\mathcal{X}/C}\rightarrow0.$$
\end{proof}

By Lemma \ref{large}, one sees that $\mathcal{O}_{\mathcal{X}}(H)$
and $\mathcal{O}_{\mathcal{X}}(K_{\mathcal{X}}+H)[1]$ are
$\nu_{H,F}^{\alpha, \beta}$-stable objects in $\Coh_C^{\beta
H}(\mathcal{X})$ for any $\alpha>0$ and $-2< \beta<1$. Hence from
Lemma \ref{t-stability}, it follows that there is a non-negative
rational number $t_0$ such that $\mathcal{O}_{\mathcal{X}}(H)$ and
$\mathcal{O}_{\mathcal{X}}(K_{\mathcal{X}}+H)[1]$ are $\nu_{\alpha,
\beta, t}$-stable for any $t\geq t_0$, $-2< \beta<1$ and $\alpha>0$.

\begin{theorem}\label{main2}
Assume that the following inequalities hold:
\begin{enumerate}
\item $\alpha-2<\beta<1-\alpha$;
\item $t>\max\{\frac{-\beta(\beta+2)H^3+4(\beta+2)(g-1)+\alpha^2}{(\beta+2)^2-\alpha^2}, t_0\}$.
\end{enumerate}
Then there exist rational numbers $a_0$, $a_1$ and $a_2$ only
depending on $\alpha$, $\beta$ and $\mathcal{X}$, such that
\begin{eqnarray}\label{6.1}
\ch_3^{\beta}(\mathcal{E})&\leq&
-\frac{\beta(\beta+1)}{2}H^2\ch^{\beta}_{1}(\mathcal{E})-(\beta
+\frac{1}{2})H\ch_2^{\beta}(\mathcal{E})+a_0\ch_0^{\beta}(\mathcal{E})
\\
\nonumber&&+a_1HF\ch_1^{\beta}(\mathcal{E})+a_2F\ch_2^{\beta}(\mathcal{E})
\end{eqnarray}
for any $\nu_{\alpha,\beta, t}$-semistable object $\mathcal{E}$ with
$\nu_{\alpha,\beta, t }(\mathcal{E})=0$.
\end{theorem}
\begin{proof}
Under our assumptions on $\alpha, \beta$ and $t$, one sees that
\begin{eqnarray*}
\nu_{\alpha,\beta,t}(\mathcal{O}_{\mathcal{X}}(H))&=&\nu_{\alpha,\beta,0}(\mathcal{O}_{\mathcal{X}}(H))+t\nu_{H,F}^{\alpha,\beta}(\mathcal{O}_{\mathcal{X}}(H))\\
&=&\frac{H^3(1-\beta)^2-\alpha^2}{2(1-\beta)}+t\frac{(1-\beta)^2-\alpha^2}{2(1-\beta)}\\
&>0&
\end{eqnarray*}
and
\begin{eqnarray*}
\nu_{\alpha,\beta,t}(\mathcal{O}_{\mathcal{X}}(K_{\mathcal{X}}+H)[1])&=&\nu_{\alpha,\beta,0}(\mathcal{O}_{\mathcal{X}}(-2H+(2g-2+H^3)F))\\
&&+t\nu_{H,F}^{\alpha,\beta}(\mathcal{O}_{\mathcal{X}}(-2H))\\
&=&\frac{H^3(\beta+2)^2-2(\beta+2)(2g-2+H^3)-\alpha^2}{2(-2-\beta)}\\
&&+t\frac{(\beta+2)^2-\alpha^2}{2(-2-\beta)}\\
&<&0.
\end{eqnarray*}
By the $\nu_{\alpha,\beta,t}$-semistability of $\mathcal{E}$, the
inequalities above imply that
$$\Hom(\mathcal{O}_{\mathcal{X}}(H), \mathcal{E})\cong\Hom(\mathcal{O}_{\mathcal{X}}, \mathcal{E}(-H))=0$$
and
$$\Ext^2(\mathcal{O}_{\mathcal{X}}, \mathcal{E}(-H))\cong\Hom(\mathcal{E}, \mathcal{O}_{\mathcal{X}}(K_{\mathcal{X}}+H)[1])=0.$$
Therefore the application of the Grothendieck-Riemann-Roch theorem
leads to
\begin{eqnarray*}
0&\geq&\chi(\mathcal{E}(-H))=\ch_3^H(\mathcal{E})
+\frac{c_1}{2}\ch_2^H(\mathcal{E})
+\frac{c^2_1+c_2}{12}\ch_1^H(\mathcal{E})+\chi(\mathcal{O}_{\mathcal{X}})\ch_0(\mathcal{E})\\
&=&\ch_3^{\beta}(\mathcal{E})+\left((\beta-1)
H+\frac{c_1}{2}\right)\ch_2^{\beta}(\mathcal{E})\\
&&+\left(\frac{(\beta-1)^2}{2}H^2+\frac{\beta-1
}{2}Hc_1+\frac{c^2_1+c_2}{12}\right)\ch_1^{\beta}(\mathcal{E})\\
&&+\left(\frac{(\beta-1)^3}{6}H^3+\frac{(\beta-1)^2}{4}H^2c_1+\frac{\beta-1
}{12}H(c_1^2+c_2)+\chi(\mathcal{O}_{\mathcal{X}})\right)\ch_0(\mathcal{E})\\
&=&\ch_3^{\beta}(\mathcal{E})+(\beta+\frac{1}{2})H\ch_2^{\beta}(\mathcal{E})+\frac{1}{2}\beta(\beta+1)H^2\ch_1^{\beta}(\mathcal{E})\\
&&-a_2F\ch_2^{\beta}(\mathcal{E})-a_1HF\ch_1^{\beta}(\mathcal{E})-a_0\ch_0(\mathcal{E}),
\end{eqnarray*}
where $c_i=c_i(T_{\mathcal{X}})$ for $i=1,2$ and the constants $a_0$, $a_1$ and $a_2$ only depend on $\alpha,
\beta$ and $\mathcal{X}$.
This completes the proof of Theorem \ref{main2}.
\end{proof}

\begin{corollary}\label{Cor}
There exist locally finite stability conditions on
$\mathcal{X}:=\mathbb{P}(\mathcal{E})$.
\end{corollary}
\begin{proof}
It turns out that the coefficient $-\frac{\beta(\beta+1)}{2}$ of $H^2\ch^{\beta}_{1}(\mathcal{E})$ in (\ref{6.1}) is positive when $-1<\beta<0$.
Therefore the conclusion follows from Theorem \ref{main} and Theorem \ref{main2}.
\end{proof}

\bibliographystyle{amsplain}

\end{document}